\documentclass[a4paper,11pt]{article}
\usepackage[utf8]{inputenc}
\usepackage{latexsym}
\usepackage{amsfonts}
\usepackage{amssymb}

\usepackage{enumitem} 
\usepackage{braket}
\usepackage{hyperref}

\usepackage[T1]{fontenc}
\usepackage{amsmath}
\usepackage{amssymb}
\usepackage{amsfonts}
\usepackage{amsthm}
\usepackage{indentfirst}
\usepackage{psfrag}
\usepackage{amscd,stmaryrd,latexsym}
\usepackage[margin=1.3in]{geometry}
\usepackage[pdftex]{graphicx}
\newtheorem{thm}{Theorem} 
\newtheorem{lem}{Lemma}[section]

\newtheorem{cor}{Corollary}

\newtheorem{Prop}{Proposition}[section]
\newtheorem*{St*}{Statement}

\newtheorem{Dfi}{Definition}

\newtheorem*{theorem*}{Theorem}
\newtheorem*{corollary*}{Corollary}
\theoremstyle{remark}

\theoremstyle{definition}

\theoremstyle{remark}
\newtheorem{oss}{Remark}[section]

\newcommand{\be}{\begin{equation}}
\newcommand{\ee}{\end{equation}}
\newcommand{\R}{\mathbb{R}}
\newcommand{\N}{\mathbb{N}}

\newcommand{\Q}{\mathbb{Q}}

\newcommand{\spt}[1]{\text{spt}\,\|#1\|}

\newcommand\res{\mathop{\hbox{\vrule height 7pt width .5pt depth 0pt
\vrule height .5pt width 6pt depth 0pt}}\nolimits}

\def\eps{\mathop{\varepsilon}}

\def\s{\sigma}

\def\Om{\Omega}
\def\om{\omega}

\def\p{\partial}

\def\eps{\mathop{\varepsilon}}

\def\Om{\Omega}
\def\om{\omega}
\def\p{\partial}

\DeclareMathAlphabet{\mathscr}{OT1}{pzc}{m}{it}

\begin{document} 

\title{\textbf{Hypersurfaces with mean curvature prescribed by an ambient function: compactness results}}
\author{Costante Bellettini\thanks{Research partially supported by the EPSRC (grant EP/S005641/1).}\\University College London}
\date{}

\maketitle

\begin{abstract}
We consider, in a first instance, the class of boundaries of sets with locally finite perimeter whose (weakly defined) mean curvature is $g \nu$, for a given continuous positive ambient function $g$, and where $\nu$ denotes the inner normal. It is well-known that taking limits in the sense of varifolds within this class is not possible in general, due to the appearence of ``hidden boundaries'', that is, portions (of positive measure with even multiplicity) on which the (weakly defined) mean curvature vanishes, so that $g$ does not prescribe the mean curvature in the limit. As a special instance of a more general result, we prove that locally uniform $L^q$-bounds on the (weakly defined) second fundamental form, for $q>1$, in addition to the customary locally uniform bounds on the perimeters, lead to a compact class of boundaries with mean curvature prescribed by $g$. 

The proof relies on treating the boundaries as \textit{oriented} integral varifolds, in order to exploit their orientability feature (that is lost when treating them as (unoriented) varifolds). Specifically, it relies on the formulation and analysis of a weak notion of \textit{curvature} coefficients for oriented integral varifolds (inspired by Hutchinson's work \cite{Hut}).  

This framework gives (with no additional effort) a compactness result for oriented integral varifolds with curvature locally bounded in $L^q$ with $q>1$ and with mean curvature prescribed by any $g\in C^0$ (in fact, the function can vary with the varifold for which it prescribes the mean curvature, as long as there is locally uniform convergence of the prescribing functions). Our notions and statements are given in a Riemannian manifold, with the oriented varifolds that need not arise as boundaries (for instance, they could come from two-sided immersions).
\end{abstract}

\section{Introduction}
\label{intro}

Let $\Om \subset \R^{n+1}$ be open; let $E_\ell\subset \R^{n+1}$, $\ell\in\N$, and $E$ be Caccioppoli sets in $\Om$ (sets with locally finite perimeter); let $g_\ell:\Om \to \R$, $\ell\in\N$, and $g:\Om\to \R$ be functions in $C^0(\Om)$. Assume that

\begin{equation}
\label{eq:bounds_assumptions_continuous}
g_\ell \to g \text{ in $C^0_{\text{loc}}(\Om)$}; \,\,\sup_\ell \int_{U} |\nabla \chi_{E_\ell}|  <\infty \text{ for any } U \subset \subset \Om;\,\, \chi_{E_\ell} \to \chi_{E} \text{ in $L^1_{\text{loc}}(\Om)$}.
\end{equation}
Note that, given the second condition in (\ref{eq:bounds_assumptions_continuous}), the existence of a Caccioppoli set $E$ for which the third condition holds is automatically true for a subsequence of $\{E_\ell\}$, by $BV$ compactness. (Also note that locally uniform $C^{0,\alpha}$-bounds on $g_\ell$ would guarantee the existence of $g$ for which the first condition holds for a subsequence of $g_\ell$.)

We denote by $\p^* F$ the reduced boundary of a Caccioppoli set $F$, that is, the set of points where the measure-theoretic unit normal exists. This set is $n$-rectifiable. The inner measure-theoretic unit normal, well-defined at points in $\p^* F$, is denoted by $\nu_{F}$. The perimeter measure $|\nabla \chi_{F}|$ (here $\nabla$ is distributional and $\chi_{F} \in \text{BV}_{loc}(\Om)$ by definition of Caccioppoli sets) coincides with $\mathcal{H}^n \res \p^*F$. These are consequences of De Giorgi's fundamental theorem. In particular, the second condition in (\ref{eq:bounds_assumptions_continuous}) can equivalently be written as $\sup_\ell \mathcal{H}^n(\p^*E_\ell \cap U)  <\infty$. We refer to \cite{Mag} for details.

\medskip

The hypersurfaces $\p^* E_\ell$ will be assumed to have mean curvature prescribed by the ambient function $g_\ell$, in a weak sense that will be specified below. (In the case in which $\p E_\ell$ and $g_\ell$ are smooth, then we will just be saying that the classical mean curvature vector of the hypersurface $\p E_\ell$ is $g_\ell \nu_{E_\ell}$, where $\nu_{E_\ell}$ is the inward unit normal.) We will be interested in the mean curvature properties of the ``limit hypersurface'' obtained from the sequence $\p^* E_\ell$; the notion of limit will be specified below. In particular we will be interested in how the mean curvature of the limit hypersurface relates to $g$. 

This problem was successfully addressed by R. Sch\"atzle in \cite{Sch}. We start with a discussion that shows why the notion of ``limit hypersurface'' and its mean curvature properties are delicate matters. An example constructed by K. Gro\ss e-Brauckmann in \cite{GB}, and reported in \cite{Sch}, gives a sequence of CMC embedded surfaces in $\R^3$ with mean curvature $1$ that converge in the Hausdorff distance to a plane $P$. The convergence can be expressed effectively in the language of varifolds, in which case the limit is the integral varifold corresponding to the plane endowed with multiplicity $2$. Varifold convergence makes area continuous (under the limit operation) and keeps track of the nature of the CMC surfaces, that are ``double-layered'' on the plane, with necks that connect the two layers and become dense as $\ell \to \infty$, see the figures in \cite[p. 376]{Sch}. Note that in this example, with reference to the notation used earlier, $g_\ell \equiv g \equiv 1$ and the set $E$ is empty ($E_\ell$ is the set ``between the two layers''). In particular, if we take limits only in the sense of boundaries, or in the sense of currents, we would just get $\p \llbracket E_\ell \rrbracket \to 0$ (as customary, $\llbracket E_\ell \rrbracket$ denotes the $(n+1)$-current of integration on $E_\ell$ and $\p$ the boundary operator on currents). This convergence does not see the accumulation of mass (area) onto the double plane. Varifold convergence, on the other hand, in particular implies that $\mathcal{H}^n \res \p^*E_\ell \rightharpoonup 2 \mathcal{H}^n \res P$ (as Radon measures). Clearly the mean curvature of $P$ is $0$, hence it is not prescribed by the ambient function $g$: the ``ambient-prescribed'' feature is lost in the limit. This is essentially due to a cancellation phenomenon: the two layers of the CMC boundaries approach onto $P$ with mean curvature vectors that are close to being opposite on a large set; this set gets larger and larger as $\ell \to \infty$, and on it the two vectors get closer and closer to being opposite.

\medskip

When we look at the convergence of $\mathcal{H}^n \res \p^*E_\ell \rightharpoonup \|V\|$, where $\|V\|$ denotes the weight measure of the limit varifold, it is immediate to see that $\Sigma=\text{spt}(\|V\|)$ contains $\overline{\p^* E}$. As shown by the previous example, where $\p^* E = \emptyset$ and $\|V\|=2 \mathcal{H}^n \res P$, the strict inclusion is possible. It is also fairly straightforward to check that on $\text{spt}(\mu_V) \setminus \p^* E$ the mean curvature is almost everywhere $0$ and that the multiplicity is almost everywhere even on $\text{spt}(\|V\|) \setminus \p^* E$. A priori it could also happen that the multiplicity of the limit varifold on (a set of positive measure in) $\p^* E$ is odd and higher than $1$. This is however shown to be impossible in \cite{Sch} if $g\neq 0$. In other words, if $g$ does not vanish then the only difference between the limit in the sense of boundaries and the varifold limit is given by ``even-multiplicity minimal pieces'', that appear only in the varifold limit. Any such portion is known in the literature as ``hidden boundary''.

\medskip

As a relevant special instance of Theorem \ref{thm:main} below, we have the following compactness result, valid for a class of boundaries with mean curvature prescribed by an ambient function that does not vanish (or only vanishes on a sufficiently small set). Given the above example, an additional assumption is to be expected, to avoid hidden boundaries. It is known that hidden boundaries do not appear if the Morse index of the hypersurfaces is uniformly bounded; this follows from the (more general) results in \cite{BW1} \cite{BW2}, \cite{SchSim}\footnote{Digressing a bit, the `cancellation of first variation' under study arises in other situations as well, see \cite[Section 1.2]{Sch}. For example, in the context (\cite{HT}) of the Allen--Cahn inhomogeneous PDE, a sequence of solutions with first variation equal to a given positive constant can give rise (as the Allen--Cahn parameter $\eps \to 0$) to a hyperplane with multiplicity $2$ (vanishing first variation). In this case one wants to relate first variations with respect to different functionals (Allen--Cahn for the sequence, area for the limit) and the phenomenon is likely to be more complex. Even assuming uniform Morse index bounds on the sequence, it is at present unknown whether minimal limits can arise (the optimal regularity result for the limit is on the other hand available under index bounds, see \cite[Theorem 1.2]{BW3}).}. The scope here is to identify a weaker assumption that guarantees compactness (in particular, the preservation of the 'prescribed mean curvature condition'). We identify it in an $L^q$ bound on the curvatures, for $q>1$. 

\begin{thm}
\label{thm:C^2_case}
Let $g_\ell, g$ ($\ell\in \N$) be functions and $E_\ell, E$ be sets of locally finite perimeter in $\R^{n+1}$ that satisfy (\ref{eq:bounds_assumptions_continuous}). Assume that $\mathcal{H}^n(\{g=0\})=0$ (in particular, $g \geq 0$ or $g \leq 0$) and that $\p E_\ell = \p^* E_\ell$ are $C^2$-embedded for all $\ell$. Denoting by $\nu_\ell$ the inner unit normal to $\p E_\ell$, assume that the mean curvature vector of $\p E_\ell$ is given by $g_\ell \nu_\ell$ for all $\ell\in \N$.

\noindent Denote by $B^{\ell}$ the second fundamental form of the hypersurface $\p E_\ell$. Assume that for $q>1$ we have, for every open set $U \subset \subset \Om$,
$$\sup_{\ell\in \N} \int_{\p E_\ell \cap U} |B^{\ell}|^q \,d\mathcal{H}^n<\infty.$$
Then the set of locally finite perimeter $E$ in (\ref{eq:bounds_assumptions_continuous}) satisfies: 
\begin{itemize}
 \item $|\p^* E_\ell|$ converges to $|\p^* E|$ in the sense of varifolds (as customary, $|\p^* E|$ is the multiplicity-$1$ varifold associated to the reduced boundary of $E$);
 \item the first variation of $|\p^* E|$ (as a functional on $C^1_c$-vector fields) is represented by integration (with respect to $\mathcal{H}^n \res \p^* E$) of $g \nu$, where $\nu$ is the ($\mathcal{H}^n$-a.e.~defined) inward pointing unit normal to $\p^* E$.
\end{itemize}
\end{thm}

\begin{oss}
It follows from the more general Theorem \ref{thm:main} that $C^2$-embeddedness can be replaced by $C^2$-quasi-embeddedness, in the sense of \cite{BW2}.
\end{oss}

The two conclusions of Theorem \ref{thm:C^2_case} state respectively that 'hidden boundaries' do not arise and that the 'prescribed mean curvature condition' is preserved in the limit. In fact, it is known from the main result in \cite{Sch} (which makes no curvature hypotheses) that when $g\neq 0$ the loss of the 'prescribed mean curvature condition' for a sequence of boundaries arises exactly where 'hidden boundaries' appear (and viceversa). The proof of Theorem \ref{thm:C^2_case} given here is independent of \cite{Sch}.

We treat the hypersurfaces in question neither as integral varifolds nor as boundaries, but rather as \textit{oriented} integral varifolds. This notion was introduced by J. E. Hutchinson in \cite{Hut}, which also provides the basic compactness result for this class (details in Section \ref{oriented_var}). We informally sketch some key ideas here: first we analyse the example of CMC boundaries given earlier from this different point of view.

\medskip

Viewing the smooth boundary $\p E_\ell$ as an oriented integral varifold means that we ``lift'' it to the oriented Grassmannian $\Om \times \mathbb{S}^n$ (as opposed to the unoriented Grassmannian, which corresponds to the more customary notion of (unoriented) varifold). The point $x \in \p E_\ell$ is lifted to $(x, \nu_{E_\ell}(x))$. More precisely, the oriented integral varifold associated to the smooth oriented hypersurface $\p E_\ell$ is the Radon measure $V^\ell$ in $\Om \times \mathbb{S}^n \subset \Om \times \R^{n+1}$ that is supported on the $n$-dimensional submanifold $\{(x, \nu_{E_\ell}(x)): x\in \p E_\ell\}$ and is characterised by the fact that its pushforward via the projection onto the first factor $(x,v)\in \Om \times \R^{n+1} \to x \in \Om$ is the measure $\mathcal{H}^n \res \p E_\ell$. In the example in question, assuming that $P$ is the plane spanned by $\textbf{e}_1, \textbf{e}_2$, taking the limit in the sense of oriented varifold means that we pass to the limit (as measures in $\R^{n+1} \times \R^{n+1}$) the lifts considered; we obtain, as limit, $V=\mathcal{H}^2 \res P_+ + \mathcal{H}^2 \res P_-$, where $P_+=\{(x,v):x\in P, v=\textbf{e}_3\}$ and $P_-=(x,v):x\in P, v=-\textbf{e}_3\}$. Taking the limit as oriented varifolds keeps track of the fact that the double plane arises as the union of two copies of $P$ with opposite orientations (respectively $\textbf{e}_3$ and $-\textbf{e}_3$). In the notation that will be introduced in Section \ref{prelim}, $V^\ell = (\p E_\ell, \nu_{E_{\ell}}, 1, 0)$ and $V=(P, \textbf{e}_3, 1, 1)$.

The key advantage of this is that, not only we can speak of the unit normal in the limit, but crucially the unit normal is (the restriction of) a \textit{fixed} \textit{ambient} function: it is the vector valued function $\vec{v}$, where $v=(v_1, \ldots, v_{n+1})$ is the variable in the second factor of $\Om \times \R^{n+1}$, in which we embed the oriented Grassmannian bundle $\Om \times \mathbb{S}^{n}$.

Given the prescribed nature of the mean curvature for each $\ell$, the mean curvature vector can also be lifted to $\Om \times \R^{n+1}$ as the \textit{vector-valued} function $\vec{\mathscr{g}}_\ell(x,v)=g_\ell(x)v$. This is a (continuous) ambient function, generally depending on $\ell$. Under hypothesis (\ref{eq:bounds_assumptions_continuous}) the sequence $\vec{\mathscr{g}}_\ell$ converges to $\vec{\mathscr{g}}(x,v)=g(x) v$ in $C^0_\text{loc}(\Om \times \R^{n+1})$. (This is independent of the sequence $V^\ell$.) In the specific example, $g_\ell=g\equiv 1$. The issue of ``loss of prescribed mean curvature condition'' corresponds to the fact that the candidate vector field $\vec{\mathscr{g}}$ does not yield the mean curvature vector on the limit $V$.
In the example in question, the candidate is $v$; its restriction gives $\textbf{e}_3$ on $P^+$ and $-\textbf{e}_3$ on $P^-$, which are clearly not the mean curvature vectors of $P$ with orientation (respectively) $\textbf{e}_3$ and $-\textbf{e}_3$. One can note that the average of the candidate mean curvature vectors is $0$, which is the correct mean curvature of $P$, with either orientation; this is a general fact that we will observe in due course, and it corresponds to the cancellation effect of the mean curvature vectors that happens when hidden boundaries form, as discussed earlier.

This suggests that writing the usual first variation formula (in which one uses a vector field in $\Om$) for $V^\ell$, as an identity in the oriented Grassmannian bundle, and passing it to the limit yields an identity that is not separately meaningful for portions that have opposite orientations, like  $P^+$ and $P^-$. For $V$, $\vec{\mathscr{g}}$ itself has no variational meaning, only its average (in $v$) does; in fact, one can write an analogous (valid) identity for $V$, replacing $\vec{\mathscr{g}}$ with any vector-valued function that has the same average values (see Section \ref{lift_mean}). To obtain separate first variation information on $P^+$ and $P^-$, one would like to evaluate the first variation on a vector field that deforms $P^+$ and $P^-$ independently, which is not possible with a vector field in $\Om$: one needs to employ a vector field that \textit{depends on} the oriented unit normal (a function of $(x,v)$). This suggests that full curvatures should be relevant, in order to argue that the first variation operator for the sequence gives rise to the first variation operator for the limit. 

Again in Hutchinson's work \cite{Hut}, the notion of 'curvature integral varifolds' is introduced; this amounts to the fact that the (unoriented) integral varifold admits a weak notion of second fundamental form, or curvature coefficients, defined via the validity of a certain identity in the unoriented Grassmannian bundle. (See also C. Mantegazza \cite{Man}, where the notion is extended to include the possible presence of boundary.) In the absence of any regularity information in the limit, it is natural to employ a weak notion of mean curvature, and, more generally, a weak notion of curvature coefficients.

We will revisit the definition of curvature varifolds given in \cite{Hut}, formulating one that is more suited (for our purposes) to the \textit{oriented} class (details in Section \ref{curv_oriented_var}). The identity that defines the weak notion of curvature coefficients $W_{ab}^{\ell}$ associated to $\p^* E_{\ell}$ holds in $\Om \times \R^{n+1}$ (in which we embed the oriented Grassmannian bundle $\Om \times \mathbb{S}^n$). This identity can be passed to the limit thanks to the $L^q$ bounds assumed. In view of the discussion above on the mean curvature, it is not a priori clear that the limit coefficients $W_{ab}$ appearing in this identity are ``meaningful curvature coefficients'' for the limit oriented integral varifold $V$. 
However, it turns out that the identity for weak second fundamental form of oriented integral varifolds is rigid (unlike the one obtained from the first variation formula alone): the definition is well-posed, more precisely, if curvature coefficients exist, there is a unique choice and they must be odd in $v$. 

We check this by relating the coefficients $W_{ab}$, that are in principle just functions in $L^q(V)$ satisfying an identity, to a ``geometric'' notion of curvature coefficients (Section \ref{geometric_coeff}). To make sense of this, we exploit the fundamental $C^2$-rectifiability of $\|V\|$ given by U. Menne \cite{Men}, which provides countably many $C^2$ hypersurfaces that cover $\|V\|$-almost every point. We use these $C^2$ hypersurfaces to build appropriate test functions in $\Om \times \R^{n+1}$, to be used in the identity.

The parity property removes the issue that was arising earlier
and implies that the candidate mean curvature $\vec{\mathscr{g}}$ is the correct one (Section \ref{proofs}).
Having established that the prescribed mean curvature condition holds in the limit, the fact that the limit is a boundary follows by fairly straighforward arguments. 

\medskip

In view of the ideas described, it is natural to prove a more general result that employs the framework of ``oriented integral varifolds with curvature'' in the assumptions as well. This allows in particular to remove the $C^2$ assumption in Theorem \ref{thm:C^2_case} (which was made to make sense classically of the second fundamental form). The result is a general compactness statement with preservation of the 'prescribed mean curvature condition'; we will give the most general form in Theorem \ref{thm:overall}, once all the relevant notions have been introduced. For the moment, we record the following fairly general instance, which immediately implies Theorem \ref{thm:C^2_case}. Note that $g\in C^0$ is unrestricted, and the limit is not necessarily a boundary (as easy examples show), it is just an oriented integral varifold $V$, and the multiplicity of $\|V\|$ is a.e.~equal to $1$ where $g\neq 0$.

\begin{thm}
\label{thm:main}
Let $g_\ell, g$ ($\ell\in \N$) be functions and $E_\ell, E$ be sets of locally finite perimeter in $\R^{n+1}$ that satisfy (\ref{eq:bounds_assumptions_continuous}). Denoting by $\nu_\ell$ the inner unit normal to $\p^* E_\ell$ (in the sense of De Giorgi), assume that the first variation of $|\p^* E_\ell|$ is given by integration (with respect to $\mathcal{H}^n \res \p^* E_\ell$) of $g_\ell \nu_\ell$ for all $\ell\in \N$. 
Assume that, for every $\ell$, the oriented integral varifold $V^{\ell}=(\p^* E_\ell, \nu_\ell, 1, 0)$ has curvature in the sense of Definition \ref{Dfi:oriented_curvature_varifold} below, with curvature coefficients $W^{\ell}_{ia}$, for $a, i \in \{1, \ldots, n+1\}$. Moreover, let $q>1$ and assume that, for every open set $U \subset \subset \Om$, 
$$\sup_{\ell\in \N} \int_U |W^{\ell}_{ia}|^q \,dV^{\ell}<\infty.$$

Then there exists an oriented integral varifold $V$, to which $V^\ell$ converge (in the sense of oriented varifolds) such that 

\begin{itemize}
 \item  
 $V=(\p^*E, \nu_E, 1, 0) + (\mathcal{R}, \textbf{n}, \theta, \theta)$, where $\mathcal{R}$ is $n$-rectifiable,  $\textbf{n}$ is any measurable determination of unit normal a.e.~on $\mathcal{R}$, $\theta: \mathcal{R}  \to \N\setminus \{0\}$ is $L^1_\text{loc}(\mathcal{H}^n \res \mathcal{R})$ and $\mathcal{R}\subset \{x:g(x)=0\}$;
 \item the generalised mean curvature of $\|V\|$ is given ($\|V\|$-a.e.) by $g \nu_E$ on $\p^* E$ and by $g \textbf{n} = -g \textbf{n} =0$ on $\mathcal{R}$ (as $\|V\|$ has automatically first variation that is an absolutely continuous measure with respect to $\|V\|$, the generalised mean curvature represents the first variation of $n$-area of $\|V\|$).
\end{itemize}
\end{thm}

The first conclusion implies that $\Theta(\|V\|, x)=1$ for $\|V\|$-a.e.~$x\in\R^{n+1}\setminus \{x:g(x)=0\}$. Note that $\mathcal{H}^n \res(\p^*E \cap \{g\neq 0\}) = \|V\|\res \{g\neq 0\}$.

\begin{oss}
Part of the conclusions in Theorem \ref{thm:main} follows from \cite{Sch}, at least if we assume $g_\ell$ to satisfy locally uniform $C^{1}$ bounds. Namely \cite{Sch} gives (even without curvature assumptions) that the (unoriented) varifold obtained in the limit is of the type $|\p^* E| + (\mathcal{R}, 2\theta)$ and that the generalised mean curvature of this varifold is (a.e.) given by $g \nu_E$ on $\p^* E$ and $0$ on $\mathcal{R}$ (in particular, $\mathcal{R}\cap \p^* E\subset \{g=0\}$, up to an $\mathcal{H}^n$-negligeable set). The relevant additional conclusion in Theorem \ref{thm:main} is that $\mathcal{R}$ must be contained in $\{g=0\}$, thereby preserving the ``prescribed mean curvature'' feature. (This is not true without the curvature assumption.)  We will not rely on \cite{Sch} to obtain the partial conclusions just recalled, as our proof will give all the conclusions at once.
\end{oss}

\medskip

Under suitable restrictions on the nodal set $\{g=0\}$ of $g$, Theorem \ref{thm:main} becomes a compactness result in the class of boundaries. This was the case in Theorem \ref{thm:C^2_case}, by dimensional considerations. We provide another instance in the following corollary (whose proof is immediate, once Theorem \ref{thm:main} is proved).

\begin{cor}
\label{cor:boundaries}
Under the assumptions of Theorem \ref{thm:main}, further assume that $g$ is a smooth function such that $S=\{g=0\}$ is a smooth hypersurface with the property that its mean curvature $h_S$ can only vanish (on $S$) to finite order. Then $V=(\p^* E, \nu_E, 1, 0)$, equivalently, $|\p^* E_\ell| \to |\p^* E|$.
\end{cor}

Theorems \ref{thm:C^2_case}, \ref{thm:main} and Corollary \ref{cor:boundaries} continue to hold if $\Om$ is a Riemannian manifold; we prove this in Section \ref{Riemannian}, Theorem \ref{thm:Riemannian}, after giving the relevant definitions. We note that the class of functions considered in Corollary \ref{cor:boundaries} is generic in the set of smooth functions (in the sense of Baire category) on a compact Riemannian manifold, see \cite[Proposition 3.8]{ZZ}. 

\medskip

\textit{\textbf{A more general class}}. As mentioned earlier, the arguments employed in this work lead to a compactness result, Theorem \ref{thm:overall} in Section \ref{Riemannian}, that permits to pass to the limit the prescribed-mean-curvature condition, regardless of whether the oriented varifolds in question arise as boundaries. (As a special instance we address the case of a sequence of two-sided immersions with mean curvature prescribed by ambient functions, see Theorem \ref{thm:immersions}.) Looking at the problem in this generality has the effect of decoupling the `loss of mean curvature' from the `formation of multiplicity'. Indeed, multiplicity may become higher than $1$ in the region $\{g\neq 0\}$ (unlike in the case of boundaries); the prescribed-mean-curvature condition is nonetheless preserved (with respect to a measurable unit normal on the limit).

\medskip

\textit{\textbf{Regularity aspects}}. We do not address here the question of optimal regularity of the limit varifold. Some regularity conclusions in Theorem \ref{thm:C^2_case} (and, restricting to $\{g\neq 0\}$, in Theorem \ref{thm:main}) are immediate from De Giorgi's and Allard's theorems. Indeed, $\|V\|=\mathcal{H}^n\res \p^* E$ and at every $x\in \p^* E$ the density satisfies $\Theta(\|V\|, x)=1$; as the first variation of $V$ is given by integration of $g \nu_E \in L^\infty(\|V\|;\R^{n+1})$, for every $x\in \p^* E$ there exists an open neighbourhood of $x$ in which $\spt{V}$ is a $C^{1,\alpha}$ (embedded) hypersurface (for every $\alpha\in (0,1)$), and in which $\spt{V}=\p^* E$. As $\mathcal{H}^n(\spt{V}\setminus \p^* E)=0$ we conclude that there exists a closed set $N$ with $\mathcal{H}^n(N)=0$ such that $\spt{V}\setminus N$ is $C^{1,\alpha}$, for every $\alpha \in (0,1)$. (If $g\in C^{k,a}$, the regularity of $\spt{V}\setminus N$ is improved to $C^{k+2,a}$ by standard Schauder theory.)

This conclusion does not really make use of the curvature bounds and finer regularity results can probably be obtained. To give an explicit idea, the following appears to be a plausible conjectural statement (not the most ambitious one), to which we plan to return to in future work. In Theorem 1 assume further that $q=2$, $g_\ell \to g$ in $C^{0,1}_{\text{loc}}(\Om)$, and $g>0$. Then $V=|\p^* E|$ is the varifold of integration over a $C^{2,\alpha}$ embedded hypersurface $M$, and $\text{dim}\left(\spt{V}\setminus M\right)\leq n-2$.

A simple example of this situation can be obtained in $U\subset \R^3$, for $g_\ell\equiv g \equiv 1$, by considering a sequence of (portions of) unduloids $M_\ell$, with constant mean curvature $1$, that converge to the union of two spherical hemispheres $S_1, S_2$ (with constant mean curvature $1$) that intersect tangentially at a point $p=S_1\cap S_2$. We note that, in this case, $M_\ell$ have uniformly bounded areas and Morse index equal to $1$ (for the functional $\mathcal{H}^2(\p^* F \cap U) - \mathcal{H}^3(F \cap U)$). This implies uniform $L^2$ bounds on the second fundamental forms. 
(Uniform mass bounds, mean curvature prescribed by an ambient function and uniform Morse index bounds are the setting for the regularity theory of \cite{BW1}, \cite{BW2}.)

On the other hand, we may consider the hypersurfaces $M_\ell \times \R^k$ in $\R^{3+k}=\R^{n+1}$. They still satisfy that the masses are locally uniformly bounded, their mean curvatures are constant and equal to $1$, and the second fundamental forms are locally uniformly bounded in $L^2$. These hypersurfaces converge to $(S_1 \times \R^k)\cup (S_2 \times \R^k)$. In this example, the Morse index of $M_\ell \times \R^k$ tends to infinity (with $\ell$) in any bounded open set that intersects $\{p\}\times \R^k$. 

\section{Preliminaries}
\label{prelim}

\subsection{Oriented integral varifolds}
\label{oriented_var}
We recall the notion of oriented integral varifold from \cite{Hut}, limiting ourselves to the codimension-$1$ case that will be of interest here. (For the moment, in Euclidean space --- see Section \ref{Riemannian} for the Riemannian setting.) We will identify $\mathbb{S}^n$ with the oriented Grassmannian of $n$-planes in $\R^{n+1}$. The identification is given by the Hodge star operator $\star$; for $v\in \mathbb{S}^n$, $\star v$ is the unit $n$-vector that spans the orthogonal to $v$ and whose orientation is such that $\star v \wedge v$ agrees with the orientation of $\R^{n+1}$.

Let $U\subset \R^{n+1}$ be open. A Radon measure $V$ in the oriented Grassmannian bundle $U\times \mathbb{S}^n$, that we will always view as embedded in $U\times \R^{n+1}$, is called an \textit{oriented $n$-varifold}. Such a $V$ is an oriented \textit{integral} $n$-varifold if there exist an $n$-rectifiable set $R$ in $U$, a couple $(\theta_1, \theta_2)$ of locally integrable $\N$-valued functions on $R$, with $(\theta_1, \theta_2) \neq (0,0)$, and a measurable choice of orientation $\xi:R \to \R^{n+1}$ (the latter means that, for $\mathcal{H}^n$-a.e.~$x\in R$, $\xi(x)\in\mathbb{S}^n \subset \R^{n+1}$ is one of the two choices of unit normal to the approximate tangent at $R$, and the map $\xi$ is measurable) such that $V$ acts on $\varphi\in C^0_c(U\times \R^{n+1})$ as follows:

\begin{equation}
\label{eq:action_oriented_var}
V(\varphi) = \int_R \theta_1(x) \varphi(x,\xi(x))) + \theta_2(x) \varphi(x,-\xi(x))) \,\, d\mathcal{H}^n(x).
\end{equation}
We write $V=(R, \xi, \theta_1, \theta_2)$ in the above situation. The class of oriented integral $n$-varifolds just introduced is denoted by $IV^o_n(U)$ (following the notation in \cite{Hut}).

\medskip

The $n$-current $\textbf{c}(V)$ associated to an oriented integral $n$-varifold $V$ is defined as follows by its action on an arbitrary $n$-form $\eta$ having compact support in $U$:

$$\textbf{c}(V)(\eta)=\int \langle \star v, \eta(x) \rangle d\,V(x,v).$$
It follows that, if $V$ is an oriented integral varifold, the current $\textbf{c}(V)$ is integral as well and, with the above notation, $\textbf{c}(V)$ is the current of integration on $R$ with orientation $\xi$ and multiplicity function $\theta_1-\theta_2$.

Finally, recall that given an oriented varifold $V$, we consider $\textbf{q}:U\times \mathbb{S}^n \to U\times \mathbb{RP}^n$, the projection from the  the oriented to the unoriented Grassmannian bundle. The pushforward measure $\textbf{q}_\sharp V$ is the (unoriented) varifold associated to $V$. With slight abuse of notation, we will identify $\mathbb{RP}^n$ with the Grassmannian $G_n(n+1)$ of unoriented $n$-planes in $\R^{n+1}$, which is in turn identified with the $(n+1)\times (n+1)$-matrices of orthogonal projection onto said plane; the latter is a submanifold of $\R^{(n+1)^2}$. Concretely, for $v \in \mathbb{S}^n \subset \R^{n+1}$ we have $\textbf{q}(v)= \delta_{ij}-v_i v_j$, for $i,j\in \{1, \ldots, n+1\}$; the $(n+1)\times (n+1)$-matrix $\delta_{ij}-v_i v_j$ can be identified with a point in $\R^{(n+1)^2}$ using its entries as coordinates. We will denote by $P_{ij}=\delta_{ij} - v_i v_j$ the projection matrix.

Given an integral oriented $n$-varifold $V$ as above then $\textbf{q}_\sharp V$ is the varifold of integration on $R$ with multiplicity $\theta_1+\theta_2$. The weight measure $\|\textbf{q}_\sharp V\|$ is $(\theta_1 + \theta_2) \mathcal{H}^n \res R$ and is the same as the weight measure $\|V\|$ associated to $V$. (Recall that $\|V\|$ is the pushforward of $V$ under the projection map $U\times \R^{n+1}\to U$ and $\|\textbf{q}_\sharp V\|$ is the pushforward of $\textbf{q}_\sharp V$ under the projection map $U\times \R^{(n+1)^2} \to U$.) The first variation of $\textbf{q}_\sharp V$ is denoted by $\delta (\textbf{q}_\sharp V)$ and is defined as a linear functional on $C^1_c(U; \R^{n+1})$. If the action of $\delta (\textbf{q}_\sharp V)$ is bounded by the $C^0$-norm of the vector field, then $\delta (\textbf{q}_\sharp V)$ is a measure in $U$, denoted by $\|\delta (\textbf{q}_\sharp V)\|$.

Convergence as oriented varifolds means (weak*) convergence as Radon measures in $U\times \mathbb{S}^n\subset U\times \R^{n+1}$. The following compactness result is proved by Hutchinson. 

\begin{theorem*}[compactness\footnote{The fact that there exist (subsequentially) limits that are oriented rectifiable varifolds with multiplicity taking values a.e.~in the half-integers is immediate from Allard's compactness for integral (unoriented) varifolds and from Federer-Fleming's compacntess for integral currents; in \cite{Hut} it is proved that the multiplicity must be a.e.~integer-valued.} for oriented integral varifold, \cite{Hut}]
Let $V_\ell$ be oriented integral varifolds on $U$ such that $\limsup_{\ell \to \infty}\|\textbf{q}_\sharp V\|(U)<\infty$, $\limsup_{\ell \to \infty}\|\delta (\textbf{q}_\sharp V)\|(U)<\infty$ and $\limsup_{\ell \to \infty}M_U(\p \textbf{c}(V_\ell))<\infty$ (where $M_U$ denotes the mass of a current in $U$ and $\p$ is the boundary operator for currents). Then $V_\ell$ subsequentially converges, in the sense of oriented varifolds, to an oriented \emph{integral} varifold $V$.
\end{theorem*}

We recall that, given an (unoriented) integral varifold $W$, the first variation (of area) evaluated on $X\in C^1_c(U)$ is $(\delta W)(X) = \int \text{div}_{T} X \,dW(x,T)$, where $\text{div}_T X$ is the divergence with respect to the (unoriented) $n$-plane $T$. When $\delta W$ is a $\R^{n+1}$-valued Radon measure (in $U$), denoted by $\|\delta W\|$, its absolutely continuous part with respect to $\|W\|$ is called generalised mean curvature of $W$, denoted by $\vec{H}_W$, see \cite{SimonNotes}. (We note explicitly, in view of forthcoming sections, that this is a function defined $\|W\|$-a.e., hence in $U$, not in the Grassmannian bundle.)  

We finally notice that if $V^\ell \to V$ as oriented varifolds, then $\textbf{q}_\sharp V^\ell \to \textbf{q}_\sharp V$ as (unoriented) varifolds and $\textbf{c}(V^\ell) \to \textbf{c}(V)$ as currents. The former follows by definition of pushforward measure; the latter follows upon noticing that for any $n$-form $\eta$ with compact support in $U$, the function $\langle \star v, \eta(x) \rangle$ is in $C^0_c(U \times \R^{n+1})$.

\subsection{Lifting the first variation formula to $U \times \R^{n+1}$}
\label{lift_mean}

Let $V$ be an oriented integral $n$-varifold, $V\in IV^o_n(U)$ for $U\subset \R^{n+1}$, recall that $\textbf{q}_\sharp V$ denotes the associated (unoriented) integral $n$-varifold. Assume that the first variation $(\delta \textbf{q}_\sharp V)(X) = \int \text{div}_T X d\|V\|$ is represented by a (vector-valued) function $\vec{H}\in L^p(\|V\|)$ in the sense that, for any $X\in C^1_c(U)$, $(\delta \textbf{q}_\sharp V)(X) =-\int \vec{H} \cdot X \,d\|V\|$. Then the function $\vec{M}(x,v)=\vec{H}(x)$ is defined $V$-a.e. and automatically even, in the sense that $V$-a.e.~we have $\vec{M}(x,v)=\vec{M}(x,-v)$. Moreover, using the fact that $\|V\|$ is the pushforward of $V$, the following identity holds for every $i\in\{1, \ldots, n+1\}$ and for any $\varphi \in C^1_c(U)$, implicitly extended to $U\times \R^{n+1}$ independently of the $v$-variables (here $D_j$ denotes the partial derivative in $x_j$): 
\begin{equation}
\label{eq:lift_first_var}
\int \left((\delta_{ij}-v_i v_j) D_j \varphi + M_i \varphi \right)\,d\,V =0;
\end{equation}
this is checked by choosing the vector field $X = \varphi(x) \textbf{e}_i$ in the first variation formula and computing that $\text{div}_T (\varphi(x) \textbf{e}_i) = P_{ij}D_j \varphi=(\delta_{ij}-v_i v_j) D_j \varphi$ (with implicit summation over repeated indices). Identity (\ref{eq:lift_first_var}) is a lift to the oriented Grassmannian bundle of the first variation formula; in the situation of interest in Theorems \ref{thm:C^2_case} and \ref{thm:main}, where the generalised mean curvature of $V^{\ell}=(\p^*E_\ell, \nu_\ell, 1,0)$ is induced by an ambient function $g_\ell$,  (\ref{eq:lift_first_var}) holds with $\vec{M^\ell}(x,v)=g_\ell(x) v$ in place of $\vec{M}$ and $V^{\ell}$ in place of $V$.

We have, in Theorems \ref{thm:C^2_case} and \ref{thm:main}, a uniform mass bound and a uniform $L^\infty$ bound for the mean curvatures. 
Then we can pass to the limit, both in the sense of (unoriented) varifolds for $\textbf{q}_\sharp V^{\ell}$, by Allard's theorem, or in the sense of oriented varifolds by the result of \cite{Hut} recalled earlier (note that $\textbf{c}(V^\ell)=\p\llbracket E_\ell\rrbracket$ are cycles). We let $V$ be the oriented integral varifold obtained in the (subsequential) limit; then $\textbf{q}_\sharp V$ is the limit of $\textbf{q}_\sharp V^{\ell}$. The interest in considering (\ref{eq:lift_first_var}) is that (since we have uniform bounds on $\int |\vec{M}^\ell|^p dV$, from the $L^\infty$-bounds on $\vec{H}$), the identity can be passed to the limit, obtaining $\int \left((\delta_{ij}-v_i v_j) D_j \varphi + g v_i \varphi \right)\,d\,V =0.$
However, this does not mean that $g \vec{v}$ (appearing in the identity) is the generalised mean curvature of $\textbf{q}_\sharp V$, as shown by the example of CMC surfaces converging to a double plane (see Section \ref{intro}). In that example (with $g_\ell\equiv 1$ for every $\ell$) the limit in the sense of oriented varifold is $(P, \nu, 1, 1)$, where $\nu$ is either of the two choices of unit normal on the plane $P$. The identity (\ref{eq:lift_first_var}) holds with $M_i= v_i$ for the limit. In other words, $\vec{M}=\nu$ on the lift of $(P, \nu)$ and $\vec{M}=-\nu$ on the lift of $(P, -\nu)$. It is only the average of $\vec{M}$ that is is $0$, i.e.~the mean curvature of $P$: as we are about to see, this is not accidental.

\begin{lem}
\label{lem:mean_curv}
Let $V$ be an oriented integral varifold in $U$, $V=(R, \xi, \theta_1, \theta_2)$ with notation as in Section \ref{oriented_var}. Assume that there exists for $i\in \{1,\ldots, n+1\}$ a function $M_i$ in $L^p(V)$ such that the following identity is valid for every $\varphi \in C^1_c(U)$ ($\varphi$ is thought of as a function on $U \times \R^{n+1}$ that is independent of the variables in the second factor\footnote{We do not worry about the fact that $\varphi$ is formally not compactly supported in $U\times \R^{n+1}$, as the support of $V$ is in $U\times \mathbb{S}^n$, hence one can always cut-off $\varphi$ in the $v$-variables.}):

\begin{equation}
\label{eq:mean_curv_coefficients_bundle}
\int \left((\delta_{ij}-v_i v_j) D_j \varphi + M_i \varphi \right)\,d\,V =0.
\end{equation}
The first variation $\delta (\textbf{q}_\sharp V)$ is a Radon measure in $U$ that is absolutely continuous with respect to $\|V\|$; the generalised mean curvature of $\textbf{q}_\sharp V$ (the Radon-Nikodym derivative of $\delta (\textbf{q}_\sharp V)$ with respect to $\|V\|$) is given at a.e.~point in $R$ by the weighted average of the vector-valued functions $\vec{M}$ at the points $(x, \xi)$ and $(x,-\xi)$: 

$$\vec{H}(x) = \frac{\theta_1(x) \vec{M}(x,\xi(x)) + \theta_2(x) \vec{M}(x, -\xi(x))}{(\theta_1 + \theta_2)(x)}.$$
\end{lem}

\begin{proof}
For $i$ fixed, consider the measure $M_i V$ and denote by $\mu$ its pushforward by $\pi \circ \textbf{q}$. Then $\mu$ is absolutely continuous with respect to $\|V\|$. Indeed, if $A$ is such that $\|V\|(A)=0$ we also have $V(A\times \R^{n+1})=0$ and hence $(M_i V)(A\times \R^{n+1})=0$, which shows $\mu(A)=0$. We denote by $H_i$ the density of $\mu$ with respect to $\|V\|$, which gives (Radon--Nikodym theorem) $\mu = H_i \|V\|$. We claim that the first variation of $\textbf{q}_\sharp V$ is given by the Radon measure that is absolutely continuous with respect to $\|V\|$ and whose generalised mean curvature (the density of said Radon measure with respect to $\|V\|$) is the function $\vec{H}=(H_1, \ldots, H_{n+1})$. Indeed we have, for every $\varphi \in C^1_c(U)$, that the first variation of $\textbf{q}_\sharp V$ evaluated on the vector field $\varphi \textbf{e}_i$ is given by 

$$\int \text{div}_V (\varphi \textbf{e}_i) d\|V\| = \int P_{ij} D_j \varphi d\|V\|$$
and $P_{ij}=\delta_{ij}-v_i v_j$, so the last term is $\int (\delta_{ij}-v_i v_j) D_j \varphi d\,V$ and this first variation coincides with (by assumption)
$$-\int M_i \varphi \,d\,V  = -\int \varphi d(\pi\circ \textbf{q})_\sharp (M_i V) = -\int H_i \varphi d\|V\| =  -\int \vec{H} \cdot (\varphi \textbf{e}_i) \,d\,\|V\| . $$
By linearity of the first variation, we conclude that for any vector field $X\in C^1_c(U)$ the first variation of $V$ is given by $ -\int \vec{H} \cdot X \,d\,\|V\|$. Note that $\vec{H}\in L^p(\|V\|)$.

To complete the proof of the lemma, note that the measure $M_i V$ is defined (by its action on any $\varphi \in C^0_c(U \times \R^{n+1})$ --- recall that the action of $V$ extends to functions in $L^p(V)$, and $M_i \varphi$ is such a function whenever $\varphi \in C^0_c$) by $(M_i V) (\varphi) = V(M_i \varphi) = \int_E \theta_1(x) M_i(x,\xi(x)))\varphi(x,\xi(x))) + \theta_2(x) (x,-\xi(x)))\varphi(x,-\xi(x))) \,\, d\mathcal{H}^n(x)$. The pushforward of $M_i V$ via $\pi\circ \textbf{q}$ is the measure $\mu$ defined (by its action on $\phi \in C^0_c(U)$) by 
$$\mu(\phi)=\int_E \left(\theta_1(x) M_i(x,\xi(x)) + \theta_2(x)M_i(x,-\xi(x))\right) \phi(x) \,\, d\mathcal{H}^n(x).$$
Since $\|V\|=(\theta_1 + \theta_2) \mathcal{H}^n \res E$ and $(\theta_1 + \theta_2)\neq 0$ a.e., we conclude that, for $\|V\|$-a.e.~$x$, $H_i(x) = \frac{\theta_1(x) M_i(x,\xi(x)) + \theta_2(x) M_i(x, -\xi(x))}{(\theta_1 + \theta_2)(x)}$.
\end{proof}
 
We note that, if $\theta_1\neq 0, \theta_2\neq 0$ for the limit $V$ of $V^\ell$, we have that the function $gv$ is odd on $V$, while $\vec{M}$ in (\ref{eq:lift_first_var}) is naturally even; hence $gv$ cannot restrict to the ``lift of the mean curvature'', unless the limit $\|V\|$ has a.e.~multiplicity $1$. (This is in agreement with \cite[p. 375, case (ii)]{Sch}.)
 
\subsection{Oriented integral varifolds with curvature}
\label{curv_oriented_var}

\begin{Dfi}
\label{Dfi:oriented_curvature_varifold}
An oriented integral $n$-varifold $V$ in $U\subset \R^{n+1}$ (recall that $V$ is a Radon measure in $U \times \mathbb{S}^{n} \subset U \times \R^{n+1}$) has curvature if there exist real-valued $V$-measurable functions $W_{i a}$, for $i, a \in \{1, \ldots, n+1\}$, defined $V$-a.e., such that the following identity holds for any $\phi \in C^1_c(U \times \R^{n+1})$ (with summation over a repeated index):
\begin{equation}
\label{eq:oriented_curvature_varifold}
\int \left((\delta_{ij}-v_i v_j) \, D_j \phi + W_{i a} \, D^*_{a} \phi  - (W_{rr} v_i + W_{ri} v_r) \phi\right) d\, V = 0.
\end{equation}
The notation $D^*_a$ denotes the partial derivative with respect to the variable $v_a$ in the second factor $\R^{n+1}$, while $D_j$ denotes the partial derivative with respect to the variable $x_j$ in $U$. We say that $V$ has curvature in $L^q$ if $W_{i a} \in L^q(V)$ for all $i,a$ and we refer to the collection $W_{ia}$ ($i, a \in \{1, \ldots, n+1\}$) as curvature coefficients.
\end{Dfi}

\begin{oss}
\label{oss:L1_firstvar}
If $V$ is as in Definition \ref{Dfi:oriented_curvature_varifold} with curvature in $L^1_{\text{loc}}(V)$, then the first variation of $\textbf{q}_\sharp V$ is representable by integration of an $L^1_{\text{loc}}(\|V\|)$ generalised mean curvature. To see this, test (\ref{eq:oriented_curvature_varifold}) with $\varphi\in C^1_c(U)$ (extended independently of the $v$-variables to $U\times \R^{n+1}$); we obtain $\int P_{is} (D_s \varphi)  - (W_{rr} v_i + W_{ri} v_r) \varphi d\, V = 0$. Then Lemma \ref{lem:mean_curv}, used with $M_i=W_{rr} v_i + W_{ri} v_r \in L^1_{\text{loc}}(V)$, implies that the first variation of $\textbf{q}_\sharp V$ is represented by integration of the average of $M_i$, which is then in $L^1_{\text{loc}}(\|V\|)$.
\end{oss}

The following argument justifies the requirement of (\ref{eq:oriented_curvature_varifold}) and shows that a $C^2$ hypersurface oriented by a choice $\nu$ of unit normal is naturally an oriented $n$-varifold that satisfies Definition \ref{Dfi:oriented_curvature_varifold}. We use the notation (as in \cite{Hut}) $\delta_i=P_{is}D_s$. Equivalently, for any function $f$, $\delta_i f$ is the $i$-th component of $\nabla^T f$, the projection of $\nabla f$ onto the tangent space, or also $\delta_i f = \nabla_{\textbf{e}_i^T} f$, where $\textbf{e}_i^T$ is the projection of $\textbf{e}_i$ onto the tangent space.

\medskip

\noindent \textit{Claim}. Let $M$ be an oriented $C^2$-embedded hypersurface in $U$, with a chosen unit normal $\nu$. Let $V$ be the oriented integral varifold associated to $(M, \nu, \theta_1, \theta_2)$. Then $V$ has curvature in the sense of Definition \ref{Dfi:oriented_curvature_varifold}, with $W_{i a}(x,v)=\delta_i (\nu_a)$ for $(x,v)$ such that $v=\nu(x)$ and $W_{i a}(x,v)=-\delta_i (\nu_a)$ for $(x,v)$ such that $v=-\nu(x)$, where $\nu_a=\nu\cdot \textbf{e}_a$. (This defines $W_{i a}$ for $V$-a.e.~$(x,v)$.)

\medskip

\noindent  The verification of the claim involves only the use of the divergence theorem, (just as in the case of (unoriented) curvature varifolds in \cite{Hut}) for the vector field $X = \phi(x,\nu(x)) \textbf{e}_i$ (where $\phi$ is as in Definition \ref{Dfi:oriented_curvature_varifold}). Denote the tangential part of $X$ by $X^T = \phi P_{ri} \textbf{e}_r$. Then
$\nabla^{M}(X^T)^r = P_{si}\frac{\p}{\p x_s} (\phi P_{ri}) \textbf{e}_i$, so that $\text{div}_{M} X^T = \textbf{e}_r \cdot \nabla^{M}(X^T)^r = P_{rs} \frac{\p}{\p x_s} (\phi P_{ri})$. Then the divergence theorem $\int_{M} \text{div}_{M} X^T d \mathcal{H}^n = 0$ gives (using $P_{rs}P_{ri}=P_{si}$ and $P_{rs}=P_{sr}$)

$$\int_{M} P_{rs} \frac{\p}{\p x_s} (\phi P_{ri}) d \mathcal{H}^n=\int_{M} P_{rs} \left(\frac{\p}{\p x_s} \phi(x,\nu(x))\right) P_{ri} + \phi \delta_r P_{ri}  =$$
$$= \int_{M} P_{is} (D_s \phi)  +  (D^*_{\ell} \phi ) \delta_{i} \nu_\ell + \phi \delta_r P_{ri}=0.$$
Note that $\delta_i P_{jk} = \delta_i ( \delta_{jk}-v_j v_k) = -(\delta_i v_j) v_k  - (\delta_i v_k) v_j$. The functions $P_{jk}$, $v_j, v_k$ are evaluated at $(x,\nu(x))$, hence $\delta_i (v_\ell(x,\nu(x))=\delta_i \nu_\ell$ (where $\nu_\ell = \nu \cdot \textbf{e}_\ell$). Therefore we can define $W_{i \ell}(x,v)=\delta_i \nu_\ell$ for $(x,v)$ such that $v=\nu(x)$. The previous identity becomes

$$\int \left((\delta_{is}-v_i v_s) \,D_s \phi + W_{i \ell}\, D^*_{\ell} \phi  - \phi (W_{rr} v_i + W_{ri} v_r)\right)\,d \mathcal{H}^n=0,$$
with the integrand evaluated at $(x, \nu(x))$. Similarly we set $W_{i \ell}(x,v)=-\delta_i \nu_\ell$ for $(x,v)$ such that $v=-\nu(x)$ and we get
$$\int \left((\delta_{is}-v_i v_s) \,D_s \phi + W_{i \ell}\, D^*_{\ell} \phi  - \phi (W_{rr} v_i + W_{ri} v_r)\right)\,d \mathcal{H}^n=0,$$ 
with the integrand evaluated at $(x, -\nu(x))$. Note that we have defined $W_{i\ell}$ $V$-a.e. A linear combination of the two identities with coefficients respectively $\theta_1, \theta_2$ gives (\ref{eq:oriented_curvature_varifold}).

\medskip

We note that, in the above example, $W_{i \ell}$ is odd in $v$, in the sense that $V$-a.e.~we have $W_{i\ell}(x,v)=-W_{i\ell}(x,-v)$. It does not seem a priori clear that the coefficients $W_{ia}$ in (\ref{eq:oriented_curvature_varifold}) are necessarily odd in $v$. A related issue is whether there could be distinct choices of $W_{ia}$ for which (\ref{eq:oriented_curvature_varifold}) holds. (As seen in Section \ref{lift_mean}, if we require the validity of (\ref{eq:oriented_curvature_varifold}) only for test functions $\varphi$ independent of $v$, then there are indeed multiple choices for $M_i$, since only the weighted average is uniquely determined.) The following result is key for our arguments and settles this issue.
 
\begin{Prop}[Curvature coefficients are unique and odd]
\label{Prop:odd_unique}
Let $V$ be an oriented integral varifold with curvature, as in Definition \ref{Dfi:oriented_curvature_varifold}, and assume that $W_{ia}\in L^1_{\text{loc}}(V)$ (for all $i,a$). Then $W_{ia}$ are odd functions, that is, $V$-a.e.~we have $W_{i a}(x,v)=-W_{i a}(x,-v)$. Moreover, any two choices of curvature coefficients in (\ref{eq:oriented_curvature_varifold}) must agree $V$-a.e. 
\end{Prop}

\begin{oss}
If we were to add the requirement that $W_{ia}$ are odd in Definition \ref{Dfi:oriented_curvature_varifold}, we would not gain any advantage in our forthcoming arguments, since the odd condition does not pass in any standard way to the limit under oriented varifold convergence.
\end{oss}

We will prove Proposition \ref{Prop:odd_unique} in Section \ref{geometric_coeff}. We further have the following properties for curvature coefficients. We note that (a) is a consequence of Proposition \ref{Prop:odd_unique}, while (c) follows from the stronger statement estalished in Section \ref{geometric_coeff}, so we postpone its proof. We will make use of (a) in the proof of Theorem \ref{thm:main} in Section \ref{proofs}; (b) and (c) are, strictly speaking, not used in the proof of Theorem \ref{thm:main} and are given here for completeness.

\begin{Prop}
\label{Prop:further_prop_curv_coeff}
Let $V$ be an oriented integral varifold with curvature, as in Definition \ref{Dfi:oriented_curvature_varifold}, and assume that $W_{ia}\in L^1_{\text{loc}}(V)$ (for all $i,a$). Then
\begin{enumerate}
 \item[(a)] the function $-(W_{rr} v_i + W_{ri} v_r)$ appearing in (\ref{eq:oriented_curvature_varifold}) is $V$-a.e.~the (even) lift of $\vec{H}\cdot \textbf{e}_i$, where $\vec{H}$ is the generalised mean curvature vector of $\textbf{q}_\sharp V$;
 
  \item[(b)] $W_{ij} v_j=0$ $V$-a.e.~;
  
  \item[(c)]  $V$-a.e. we have $W_{ij}=W_{ji}$, and $W_{ri} v_r=0$ (so that the mean curvature term in (\ref{eq:oriented_curvature_varifold}) can be written as $-W_{rr} v_i$).
\end{enumerate}
\end{Prop}

\begin{proof}
(a) As in Remark \ref{oss:L1_firstvar}, for any $\varphi\in C^1_c(U)$ (extended independently of the $v$-variables to $U\times \R^{n+1}$) we have $\int P_{is} (D_s \varphi)  - (W_{rr} v_i + W_{ri} v_r) \varphi d\, V = 0$. Proposition \ref{Prop:odd_unique} implies that $-(W_{rr} v_i + W_{ri} v_r)$ is even in $v$. By Lemma \ref{lem:mean_curv} it must then the (even) lift of $\vec{H}\cdot \textbf{e}_i$. 

(b) For any $\psi\in C^1_c(U \times \R^{n+1})$ we test in (\ref{eq:oriented_curvature_varifold}) with $\phi=(1-v_p v_p) \psi$ and obtain $\int (1-v_p v_p) P_{is} (D_s \psi) + (1-v_p v_p) W_{ia} D^*_a \psi + 2 v_a W_{ia} \psi - (W_{rr} v_i + W_{ri} v_r) (1-v_p v_p) \psi d\, V = 0$. Note that $(1-v_p v_p)$ vanishes identically on $U\times \mathbb{S}^n$ (hence on $\text{spt}(V)$), therefore $\int  v_a W_{ia} \psi  d\, V = 0$. Since $\psi$ is arbitrary, we conclude.
\end{proof}

\begin{oss}
We will prove (c), that is, $W_{ij}=W_{ji}$ $V$-a.e.,~in Section \ref{geometric_coeff}. This, combined with (b), gives also that $W_{ri} v_r$ vanishes $V$-a.e. We note that the condition $W_{ri} v_r=0$ is equivalent to the fact that the generalised mean curvature  $\vec{H}$ of $\textbf{q}_\sharp V$ is a.e.~orthogonal to the approximate tangent to $\|\textbf{q}_\sharp V\|$ (which is known to be a.e.~true, \cite[Chapter 5]{Brak}). Indeed, using (a), we have $-(P \vec{H}) \cdot \textbf{e}_a=P_{ai}  (W_{rr} v_i + W_{ri} v_r) = W_{rr}(P_{ai}v_i) + W_{ri}\delta_{ai} v_r - W_{ri} v_a v_i v_r =  W_{ra} v_r$ (using (b) and $P_{ai}v_i=0$ $V$-a.e.~). 
\end{oss}

\subsection{Relation with curvature integral varifolds in \cite{Hut}}
\label{relation_with_Hut}

While not essential for the forthcoming sections, we compare here Definition \ref{Dfi:oriented_curvature_varifold} to the notion of curvature integral varifold in \cite{Hut} (see also \cite{Man}). The validity of (\ref{eq:oriented_curvature_varifold}) for $V$ with $W_{ia} \in L^1_{\text{loc}}(V)$ implies that $\textbf{q}_\sharp V$ is a curvature integral varifold in the sense of \cite{Hut}. Indeed, let $\varphi\in C^1_c(U \times \R^{(n+1)^2})$, with variables $P_{ij}$ in the second factor, $i, j \in \{1, \ldots, n+1\}$. Then $\phi(x,v) = \varphi(x, P(v))$ with $P_{ij}(v)=\delta_{ij}-v_i v_j$ is admissible in (\ref{eq:oriented_curvature_varifold}). The chain rule gives $D^*_\ell \phi = - D^*_{jp} \varphi \,(v_j \delta_{\ell p} + v_p \delta_{\ell j})$. We now substitute in (\ref{eq:oriented_curvature_varifold}):
$$\int \left((\delta_{ij}-v_i v_j) D_j \varphi - (W_{ip} v_j + W_{ij} v_p) D^*_{jp} \varphi - (W_{rr} v_i + W_{ri} v_r) \varphi\right) d\, V = 0,$$
where $\varphi$ and its derivatives are evaluated at $(x,P(v))$. As the function $-(W_{i p}v_j + W_{ij} v_p)$ is even in $v$ (by Proposition \ref{Prop:odd_unique}), it descends to a well-defined ($\textbf{q}_\sharp V$-a.e.) function in $U \times \R^{(n+1)^2}$ that we denote by $A_{ijp}$. By definition of push-forward measure then

\begin{equation}
\label{eq:from_orient_to_unorient}
\int \left(P_{ij} D_j \varphi + A_{ijp} D^*_{jp} \varphi + A_{rir} \varphi\right) d\, \textbf{q}_\sharp V =0. 
\end{equation}
The existence of real-valued functions $A_{ijk}$, for $i, j, k \in \{1, \ldots, n+1\}$, defined $\textbf{q}_\sharp V$-a.e.~(in $U \times \R^{(n+1)^2}$) such that this identity holds expresses the fact that $\textbf{q}_\sharp V$ is a curvature integral varifold in the sense of \cite{Hut}. Note that (as expected from \cite[Proposition 5.2.4]{Hut}) $A_{ijk}=A_{ikj}$ and (using Proposition \ref{Prop:further_prop_curv_coeff} (b)) $A_{ijj}=0$.

In the above situation ($V$ is an oriented integral varifold with curvature in the sense of Definition \ref{Dfi:oriented_curvature_varifold}) the curvature coefficients $W_{ia}$ are determined by the curvature coefficients $A_{ijk}$ of $\textbf{q}_\sharp V$ by the following relation (we also include the coefficients $B_{ij}^k$ of the generalised second fundamental form of $\textbf{q}_\sharp V$ as in \cite{Hut}):
$$W_{ia}(x,v)=-P_{ad}(v)A_{idp}(x,\textbf{q}(v)) v_p = - B_{ia}^p(x,\textbf{q}(v)) v_p.$$
To see this, we observe that $A_{ijp} v_p = -W_{i p}v_j v_p - W_{ij}$, hence $P_{aj} A_{ijp} v_p = -P_{aj} W_{ij}$. Moreover, $W_{ia} = \delta_{aj}W_{ij}=P_{aj}W_{ij} + W_{ij} v_j v_a=P_{aj}W_{ij}$ (using Proposition \ref{Prop:further_prop_curv_coeff} (b)).

\medskip

While the idea (illustrated in Section \ref{curv_oriented_var}) behind Definition \ref{Dfi:oriented_curvature_varifold} is the same as for the notion of curvature (unoriented) varifolds in \cite{Hut}, to avoid confusion we remark that in \cite{Hut} the notion of curvature varifolds $\mathcal{C}V_n(U)$ is essentially given only for (unoriented) integral varifolds. In the case of an oriented integral varifold $V \in IV^o_n(U)$, \cite{Hut} defines $V\in \mathcal{C}V^o_n(U)$ (oriented curvature varifolds) by the requirement that $\textbf{q}_\sharp V \in \mathcal{C}V_n(U)$. That is, in \cite{Hut} an oriented integral varifold is said to be a curvature varifold when the associated (unoriented) varifold is a curvature varifold. In Definition \ref{Dfi:oriented_curvature_varifold} instead we require the defining identity (\ref{eq:oriented_curvature_varifold}) directly on the oriented varifold $V$, in order to say that it has curvature; the functions $W_{ij}$ are defined $V$-a.e. (thus in the oriented Grassmannian bundle $U \times \mathbb{S}^{n}\subset U\times \R^{n+1}$). 
Our notion may a priori be more restrictive than membership to $\mathcal{C}V^o_n(U)$ in the sense of \cite{Hut}.
(The two notions coincide if $V$ is sufficiently regular. For example, if $V$ is associated to a $C^2$-embedded hypersurface $M$, the condition that $V$ is an oriented curvature varifold in the sense of Definition \ref{Dfi:oriented_curvature_varifold} is equivalent to $V \in \mathcal{C}V^o_n(U)$ in the sense of \cite{Hut}.)

In the forthcoming sections we will work with oriented integral varifolds associated to reduced boundaries (without any regularity assumption). We will use the terminology ``$V$ is an oriented varifold with curvature'', or ``the oriented varifold $V$ has curvature'', to mean the condition in Definition \ref{Dfi:oriented_curvature_varifold}, without reference to the class $\mathcal{C}V^o_n(U)$ of \cite{Hut}.

\section{Curvature coefficients and their geometric counterpart}
\label{geometric_coeff}

The aim of this section (see Proposition \ref{Prop:C2_coeff} below) is to relate the coefficients $W_{ia}$ appearing in (\ref{eq:oriented_curvature_varifold}), which are just functions on $V$, to ``geometric'' curvature coefficients associated to the weight measure $\|V\|$. We assume that $W_{ia}\in L^1_{\text{loc}}(V)$ for all $i,a$. As a byproduct, we will also establish Proposition \ref{Prop:odd_unique}.
 
\medskip

We begin by recalling that $\textbf{q}_\sharp V$ is an integral varifold with first variation represented by a function in $L^1_{\text{loc}}(\|\textbf{q}_\sharp (V)\|)$, as observed in Remark \ref{oss:L1_firstvar}. The weight measures of $V$ and $\textbf{q}_\sharp (V)$ are the same, $\|V\|=\|\textbf{q}_\sharp (V)\|$. We denote by $R$ the $n$-rectifiable set in $\R^{n+1}$ such that $V=(R, \xi, \theta_1, \theta_2)$, and $\textbf{q}_\sharp (V)=(R, \theta_1+\theta_2)$.

The result of Menne \cite{Men} (see also \cite{San}, which proves a weaker version that is sufficient for the proof of Theorem \ref{thm:main}, since an $L^\infty$ bound on the generalised mean curvatures is available) guarantees the $C^2$-rectifiability of $\|V\|$, namely there exists a collection $\{M_j\}_{j=1}^\infty$ of $C^2$-hypersurfaces in $U$ such that $\|V\|(U\setminus (\cup_{j=1}^\infty M_j))=0$. 

Upon possibly redefining the hypersurfaces $M_j$, we may assume that each $M_j$ admits a continuous choice of unit normal $\nu_{M_j}$.\footnote{This can be ensured by considering, for each $j$ of the initial family $\{M_j\}$, the hypersurfaces $M_j\cap B^{n+1}_r(x)$ with $x\in U$ having rational coordinates, and with $r\in \Q$ sufficiently small to ensure that $M_j\cap B^{n+1}_r(x)$ is orientable. There are only countably many choices for $x$ and $r$, so the claim follows.}
Consider, for each $j$, the lift of $M_j$ to $U \times \mathbb{S}^{n} \subset U\times \R^{n+1}$, that is the $C^1$ hypersurface $\tilde{M}_j=\{(x,v)\in U\times \R^{n+1}: |v|=1, x\in M_j, v=\pm \nu_{M_j}(x)\}$. (Note that if $M_j$ is connected then $\tilde{M}_j$ has two connected components.)

\begin{oss}
The measure $V$ is absolutely continuous with respect to $\mathcal{H}^n\res (\cup_{j=1}^\infty \tilde{M}_j)$. 
Assume that $A\subset U\times \R^{n+1}$ satisfies $\mathcal{H}^n(A\cap (\cup_{j=1}^\infty \tilde{M}_j))=0$. Denote by $\varpi: U\times \R^{n+1} \to U$ the standard projection and by $v^{\pm}(R)$ the two sets defined respectively by
$$v^{\pm}(R)=\{(x,v)\in U\times \R^{n+1}: x\in R, \text{ the approximate tangent $T_x R$ exists, } v=\pm \xi(x)\}.$$
Then by definition of $V$ we have 
$$V(A)=(\theta_1 \mathcal{H}^n)(\varpi(A\cap v^+(R))) + (\theta_2 \mathcal{H}^n)(\varpi(A\cap v^-(R))).$$
By ($C^1$-)rectifiability, for any $j$ we have that at $\|V\|$-a.e.~$x\in M_j$ the approximate tangent to $R$ at $x$ exists and agrees with $T_x M_j$, and $\xi(x)=\pm \nu_{M_j}(x)$. In other words, for every $j$ there exists $N_j\subset M_j$ such that $\|V\|(N_j)=0$ and for every $x\in M_j \setminus N_j$ we have that $x\in R$, the approximate tangent to $R$ at $x$ exists and $(x,\pm\xi(x))\in \tilde{M}_j$. Hence $v^{\pm}(R) \setminus (N\times \R^{n+1}) \subset (\cup_{j=1}^\infty \tilde{M}_j)$, where $N=N_0 \cup (\cup_{j=1}^\infty N_j)$ and $N_0=R\setminus (\cup_{j=1}^\infty M_j)$; note that $\|V\|(N)=0$. We then have that $\varpi(A\cap v^\pm(R)) \subset \varpi(A \cap (\cup_{j=1}^\infty \tilde{M}_j)) \cup N$. This inclusion gives $V(A) \leq ((\theta_1+\theta_2) \mathcal{H}^n) \varpi(A \cap (\cup_{j=1}^\infty \tilde{M}_j))$. Recalling that the projection decreases the $\mathcal{H}^n$-measure, $\mathcal{H}^n\left(\varpi(A\cap (\cup_{j=1}^\infty \tilde{M}_j)) \right)\leq \mathcal{H}^n\left(A\cap (\cup_{j=1}^\infty \tilde{M}_j)\right)=0$, we conclude $V(A)=0$.
\end{oss}

By the previous remark and by Radon-Nikodym's theorem there exists a function $\sigma \in L^1_{\text{loc}}\left(\mathcal{H}^n\res (\cup_{j=1}^\infty \tilde{M}_j)\right)$ such that $V= \sigma\, \mathcal{H}^n \res (\cup_{j=1}^\infty \tilde{M}_j)$. This is (by definition) a rectifiable measure. By the characterisation of rectifiability through approximate tangents (e.g.~\cite[Chapter 3]{SimonNotes}), for $V$-a.e.~$(x,v)$ the approximate tangent measure to $V$ at $(x,v)$ exists and is $\sigma(x,v)\mathcal{H}^n\res \Pi_{x,v}$ for an $n$-dimensional subspace $\Pi_{x,v}$ in $\R^{n+1} \times \R^{n+1}$. Explicitly, denoting by $\eta_{(x,v),r}(\cdot)$ the dilation $\frac{\cdot - (x,v)}{r}$, this means that for $r\to 0$ the measures $\frac{1}{r^n}(\eta_{(x,v),r})_\sharp V$ converge to $\sigma(x,v)\mathcal{H}^n\res \Pi_{x,v}$, that is, for every function $\phi \in C^0_c(\eta_{(x,v),r}(U)\times \R^{n+1})$, 
$$\int \frac{1}{r^n} \phi\, d({\eta_{(x,v)}}_\sharp V) \to \s(x,v)\int \phi  \,d\mathcal{H}^n\res \Pi_{x,v}\,\,\,\, \text {as } r\to 0.$$
The above also implies (choosing $\phi$ to be a smooth, radially symmetric non-negative bump function in $\R^{2n+2}$ such that its value on the unit ball is $1$, it is decreasing in the radial variable, and its support is the ball of radius $1+\delta$, with $\delta>0$ arbitrarily small) that $\liminf_{r\to 0} \frac{1}{r^n}V(B_{(1+\delta)r})\geq \s(x,v) \omega_n$ and $\limsup_{r\to 0} \frac{1}{r^n}V(B_{r})\leq \s(x,v) (1+\delta)^n \omega_n$, hence 
\begin{equation}
\label{eq:V_r^n}
 \frac{V(B_r(x,v))}{r^n}\to \s(x,v) \omega_n. 
\end{equation}

\medskip

We next check that for $V$-a.e.~$(x,v)$ the following holds (note that $V$-a.e.~$\sigma > 0$):

\begin{equation}
\label{eq:Leb_cont_modified}
\left|\int \frac{\chi_r}{\s(x,v)} W_{ia} d\,V - W_{ia}(x,v)\right| \to 0   \,\,\,\, \text {as } r\to 0,                                                                                                                                                                                        \end{equation}
where $\chi$ is an any smooth, radially symmetric bump function in $\R^{2n+2}$ such that its integral on any $n$-plane through the origin is equal to $1$ and we use the notation $\chi_r(\cdot)=\frac{1}{r^n} \chi\left(\frac{\cdot - (x,v)}{r}\right)$. 

\noindent 
Recall that $W_{ia}\in L^1_{\text{loc}}(V)$, therefore $V$-a.e.~point is of Lebesgue continuity for $W_{ia}$, that is, for $V$-a.e.~$(x,v)$
\begin{equation}
\label{eq:Leb_cont_W}
\frac{1}{V(B_r((x,v))}\int_{B_r((x,v))} |W_{ia}-W_{ia}(x,v)| \, dV \to 0.
\end{equation}
Then 
$$\int \frac{\chi_r}{\s(x,v)} W_{ia} d\,V - W_{ia}(x,v) = \int \frac{\chi_r}{\s(x,v)} W_{ia} d\,V -  W_{ia}(x,v)  \int_{\Pi_{x,v}} \chi d\mathcal{H}^n =$$ 
$$= \int \frac{\chi_r}{\s(x,v)} \left(W_{ia}-W_{ia}(x,v)\right) d\,V+W_{ia}(x,v)\left(\underbrace{\int  \frac{\chi_r}{\s(x,v)}  d\,V -  \int_{\Pi_{x,v}}  \chi d\mathcal{H}^n}_{\to 0 \text{ as } r\to 0}\right),$$
where the last term tends to $0$ by the characterisation of tangent measure above, used with $\chi$ in place of $\phi$ (and using the definition of push-forward measure). Denoting by $C>0$ the radius of a ball that contains the support of $\chi$ we also find
$$\left|\int \frac{\chi_r}{\s(x,v)} \left(W_{ia}-W_{ia}(x,v)\right) d\,V \right|\leq \frac{\|\chi\|_{L^\infty}}{\s(x,v)}\frac{C^n}{(Cr)^n}\int_{B_{Cr}((x,v))}|W_{ia}-W_{ia}(x,v)|d\,V.$$ 
As $\frac{1}{r^n}\int_{B_r((x,v))}|W_{ia}-W_{ia}(x,v)|d\,V=\frac{V(B_r((x,v))}{r^n}\frac{1}{V(B_r((x,v))}\int_{B_r((x,v))}|W_{ia}-W_{ia}(x,v)|d\,V$ tends to $0$. By (\ref{eq:V_r^n}) and (\ref{eq:Leb_cont_W}), we conclude (\ref{eq:Leb_cont_modified}).

\medskip

Next, we point out that, for every given $j$, for $V$-a.e.~$(x,v)\in\tilde{M}_j$ it holds:
\begin{equation}
\label{eq:char_Mj}
\frac{(\chi_{\tilde{M}_j}V)(B_r(x,v))}{r^n}=\frac{(\s\,\mathcal{H}^n \res \tilde{M}_j)(B_r(x,v))}{r^n}\to \s(x,v) \omega_n\,\,\,\text{ as } r\to 0. 
\end{equation}
Indeed $\chi_{\tilde{M}_j}$ is $V$-measurable and in $L^1_{\text{loc}}(V)$, so by the Lebesgue continuity property, for $V$-a.e.~$(x,v)$ it holds $\frac{1}{V(B_r((x,v))}\int_{B_r((x,v))}\chi_{\tilde{M}_j} dV \to \chi_{\tilde{M}_j}(x,v)$ as $r\to 0$. As observed in (\ref{eq:V_r^n}), $\frac{V(B_r((x,v))}{r^n}\to \om_n  \s(x,v)$ for $V$-a.e.~$(x,v)$, so (\ref{eq:char_Mj}) follows.

\noindent Similarly, for every given $j$, and for every $(i,a)$, for $V$-a.e.~$(x,v)\in\tilde{M}_j$ it holds:
\begin{equation}
\label{eq:char_Mj_W}
\frac{1}{ \s(x,v) \omega_n r^n}\int_{B_r(x,v)} |W_{ia}\,\chi_{\tilde{M}_j} - W_{ia}(x,v)| \,d\,V\to 0\,\,\,\text{ as } r\to 0. 
\end{equation}

\medskip

We are now ready to prove the main result of this section, from which Proposition \ref{Prop:odd_unique} follows immediately.

\begin{Prop}
\label{Prop:C2_coeff}
Let $V$ be an oriented integral varifold with curvature in $L^1_{\text{loc}}(V)$ (as in Definition \ref{Dfi:oriented_curvature_varifold})  and let $\{M_j\}_{j=1}^\infty$, $\{\tilde{M}_j\}_{j=1}^\infty$ be as above. For any $j$, for $V$-a.e.~$(x,v)\in \tilde{M}_j$ the curvature coefficients $W_{ia}$ appearing in (\ref{eq:oriented_curvature_varifold}) are given ($i, a \in \{1, \ldots, n+1\}$) by $W_{ia}((x,v))=\delta_i (\nu_{M_j})_a$ for $v=\nu_{M_j}(x)$ and $W_{ia}((x,v))=-\delta_i (\nu_{M_j})_a$ for $v=-\nu_{M_j}(x)$.
\end{Prop}

\begin{oss}
 \label{oss:well_posed_curv_coeff}
As shown earlier, $V$-a.e.~point is in at least one of the $\tilde{M}_j$'s, so the conclusion covers indeed $V$-a.e.~point. It also follows from the proof that the conclusion is $V$-a.e.~well-posed (even though a point may belong to more than one of the $\tilde{M}_j$'s).
\end{oss}

\begin{proof}
We are going to prove the statement at an arbitrary $(x_0,v_0)\in \tilde{M}_j$ such that $T_{x_0} R$ exists, $T_{x_0} M_j= T_{x_0} R$, $\pm v_0$ is orthogonal to $T_{x_0} M_j$ at $x_0$, and such that (\ref{eq:V_r^n}), (\ref{eq:Leb_cont_modified}), (\ref{eq:Leb_cont_W}), (\ref{eq:char_Mj}) and (\ref{eq:char_Mj_W}) hold with $x_0$, $v_0$ in place respectively of $x$, $v$. These conditions are all verified at $V$-a.e.~point in $\tilde{M}_j$, as checked above.

Since we have either $v_0=\nu_{M_j}(x_0)$ or $v_0=-\nu_{M_j}(x_0)$, we begin with the first of these two instances. Let $f$ be a $C^2$ function defined in a neighbourhood $O\subset U$ of $x_0$, chosen such that $M_j\cap O=\{f=0\}$, $\nabla f \neq 0$ in $O$ and $\frac{\nabla f}{|\nabla f|}=\nu_{M_j}$. Define $\phi_a:O\times \R^{n+1}\to \R$ by
$$\phi_a(x,v) = v_a - \frac{D_a f}{|\nabla f|}(x).$$
Let $\tilde{\psi}:\R^{n+1}\to [0,1]$ be a smooth function (of $v$) that is identically $1$ in an open neighbourhood of the set $\{v: v=\nu_{M_j}(p),\, p\in M_j\cap O\}$ and identically $0$ in an open neighbourhood of the set $\{v: v=-\nu_{M_j}(p),\, p\in M_j \cap O\}$. The fact that these two sets are a positive distance apart can be ensured by taking $O$ sufficently small. We define $\psi(x,v)=\tilde{\psi}(v)$ on $O\times \R^{n+1}$. We then check that
\begin{equation}
\label{eq:vanishing_integral_f}
\frac{1}{r^{n+1}}\int_{B_r((x_0,v_0))}|\psi \phi_a| d\,V \to 0 \text{ as }r\to 0.
\end{equation} 
Indeed, $\psi \phi_a$ vanishes on $\tilde{M}_j$ by construction ($\phi_a$ vanishes on $\tilde{M}_j\cap \{\psi\neq 0\}$). Moreover (by $C^1$-regularity of $\psi \phi_a$), there exists $C>0$ such that $|\psi \phi_a|\leq C r$ on $B_r((x_0,v_0))$ for all sufficently small $r$. Therefore (\ref{eq:vanishing_integral_f}) follows from the fact that $\lim_{r\to 0}\frac{1}{r^{n}} V(B_r((x_0,v_0)) \setminus \tilde{M}_j) =0$, which is given by the assumed validity of (\ref{eq:V_r^n}) and (\ref{eq:char_Mj}) at $(x_0, v_0)$. Similarly, we have 
\begin{equation}
\label{eq:vanishing_integral_f_2}
\frac{1}{r^{n+1}}\int_{B_r((x_0,v_0))}|D_\ell^* \psi|\, |\phi_a| d\,V \to 0 \text{ as }r\to 0\,\, \text{ for every }\ell.
\end{equation}

Testing (\ref{eq:oriented_curvature_varifold}) with $\psi  \beta \phi_a$ (in place of $\varphi$), for an arbitrary $\beta\in C^1_c(O\times \R^{n+1})$, we find

\begin{eqnarray*}
\int  (\delta_{is}-v_i v_s) \beta \psi \,\frac{\p}{\p x_s} \left(\frac{-D_a f}{|\nabla f|}\right) \,d\,V+ \int (\delta_{is}-v_i v_s)\,(D_s\beta) \psi\,\phi_a \,d\,V + \int W_{i a}\,\beta \psi\,d\,V\\
+\int W_{i\ell}\phi_a D^*_\ell (\beta \psi) \,d\,V- \int \beta \psi \phi_a (W_{rr} v_i + W_{ri} v_r)d\,V=0.
\end{eqnarray*}
We let $\beta=\frac{\chi_r}{\sigma(x_0,v_0)}$ in the above identity, where $\chi$ and $\chi_r$ are as described for (\ref{eq:Leb_cont_modified}), and analyse the five terms. Since $|D \beta|\leq \frac{K}{r^{n+1}}$ (for a constant $K$ depending on $(x_0,v_0)$), by (\ref{eq:vanishing_integral_f}) and (\ref{eq:vanishing_integral_f_2}) the second term tends to $0$ as $r\to 0$. The fourth term, as the integrand vanishes on $\tilde{M}_j$, can be replaced by $\int  W_{i\ell}\phi_a D^*_\ell (\beta \psi) (1-\chi_{\tilde{M}_j})\,d\,V$, which is bounded in modulus by $\frac{C}{r^n}\int_{B_r((x_0,v_0))} |W_{i\ell}| (1-\chi_{\tilde{M}_j})\,d\,V$, which tends to $0$ as $r\to 0$ by (\ref{eq:V_r^n}), (\ref{eq:Leb_cont_W}), (\ref{eq:char_Mj_W}). The fifth term is bounded in modulus by $\frac{Cr}{r^n}\int |W_{rr} v_i + W_{ri} v_r|d\,V$, which also tends to $0$ as $r\to 0$ by (\ref{eq:V_r^n}), (\ref{eq:Leb_cont_W}).
Since $\psi\equiv 1$ on the support of $\beta$ for all sufficiently small $r$, we obtain that the third term tends to $W_{ia}(x_0,v_0)$ by (\ref{eq:Leb_cont_modified}). The first term tends to $(\delta_{is}-(v_0)_i (v_0)_s) \frac{\p}{\p x_s} \left(\frac{D_a f}{|\nabla f|}\right)(x_0)$ (because $(x_0, v_0)$ is a point of continuity and thus of Lebesgue continuity of $(\delta_{is}-v_i v_s) \frac{\p}{\p x_s} \left(\frac{D_a f}{|\nabla f|}\right)(x)$).
In conclusion

\begin{equation}
\label{eq:coefficient_correct_1}
 W_{ia}(x_0,v_0) = (\delta_{is}-(v_0)_i (v_0)_s) \frac{\p}{\p x_s} \left(\frac{D_a f}{|\nabla f|}\right)(x_0).
\end{equation}
The second instance, that is the case in which $(x_0,v_0)$ satisfies $v_0=-\nu_{M_j}(x_0)$, is addressed in the same way, noting that $-f$ can be used in place of $f$, and it leads to
\begin{equation}
\label{eq:coefficient_correct_2}
W_{ia}(x_0,v_0) = -(\delta_{is}-(v_0)_i (v_0)_s) \frac{\p}{\p x_s} \left(\frac{D_a f}{|\nabla f|}\right)(x_0).
\end{equation}
In other words the value of $W_{ia}$ agrees with the curvature coefficient of $M_j$ computed with respect to $\nu_{M_j}$ if $v=\nu_{M_j}$ and with respect to $-\nu_{M_j}$ if $v=-\nu_{M_j}$. 
\end{proof}

\begin{oss}
Given functions $W_{ia}\in L^1_{\text{loc}}(V)$ for which (\ref{eq:oriented_curvature_varifold}) holds, $W_{ia}$ have a well-defined Lebesgue value $V$-a.e., and this is the value $W_{ia}(x_0, v_0)$ for which we have the validity of (\ref{eq:Leb_cont_modified}) in the proof just given, regardless of the chosen $j$. Therefore the relations (\ref{eq:coefficient_correct_1}) and (\ref{eq:coefficient_correct_2}) are valid with the same left-hand-side regardless of $j$. In particular this guarantees that, except possibly for a set of vanishing $V$-measure, if $(x,v) \in \tilde{M}_j \cap \tilde{M}_k$ then the curvature coefficients of $M_j$ and $M_k$ are the same at $x$. This confirms Remark \ref{oss:well_posed_curv_coeff} and is not surprising: $V$ is a rectifiable measure and so $V$-a.e.~the approximate tangent $\Pi_{(x,v)}$ exists and agrees with the approximate tangent to $T_{(x,v)} \tilde{M}_j$ independently of $j$: the tangent $T_{(x_0,v_0)} \tilde{M}_j$ determines the right-hand-sides of (\ref{eq:coefficient_correct_1}) and (\ref{eq:coefficient_correct_2}).
\end{oss}

\begin{proof}[proof of Proposition \ref{Prop:further_prop_curv_coeff} (c)]
For each $C^2$ hypersurface $M$ oriented by a unit normal $\nu_M$, we have by direct computation $\delta_i (\nu_{M} \cdot \textbf{e}_a) =- (\nabla_{\textbf{e}_i^T} \textbf{e}_{a}^T)^N = -B(\textbf{e}_i^T, \textbf{e}_{a}^T)$, where the tangential and normal components of a vector are denoted respectively using upper indeces $T$ and $N$, and $B$ is the classical second fundamental form, whose symmetry gives $\delta_i (\nu_{M} \cdot \textbf{e}_a)=\delta_a (\nu_{M} \cdot \textbf{e}_i)$. Then the conclusion follows from Proposition \ref{Prop:C2_coeff}.
\end{proof}

\section{Proof of Theorems \ref{thm:C^2_case} and \ref{thm:main}}
\label{proofs}

For each $\ell$ we have the validity of (\ref{eq:oriented_curvature_varifold}), with $V^\ell=(\p^* E_\ell, \nu_\ell, 1, 0)$ in place of $V$ and $W^\ell_{ia}$ in place of $W_{ia}$ (as assumed in Theorem \ref{thm:main}):

\begin{equation}
\label{eq:cuvature_identity_Vell}
\int \left((\delta_{is}-v_i v_s) \,D_s \phi + W^\ell_{i a}\, D^*_{a} \phi  - \phi (W^\ell_{rr} v_i + W^\ell_{ri} v_r)\right)d\,V^\ell=0.
\end{equation}

\begin{oss}
Note that Theorem \ref{thm:C^2_case} is indeed a special case of Theorem \ref{thm:main}. This follows from the argument in Section \ref{curv_oriented_var}, which shows that $V^\ell=(\p^* E_\ell, \nu_\ell, 1, 0)$ is an oriented integral varifold with curvature, with curvature coefficients $W^\ell_{ia}=\delta_i (\nu_{E_\ell})_a$. By direct computation, $\delta_i (\nu_{E_\ell} \cdot \textbf{e}_a) =- (\nabla_{\textbf{e}_i^T} \textbf{e}_{a}^T)^N$, where $T$ and $N$ stand respectively for tangential and normal components of a vector with respect to $\p E_{\ell}$. Therefore $\int \sum_{i,j} |W^\ell_{ij}|^q d V^\ell\leq \int_{\p E_\ell} |B^\ell|^q d\mathcal{H}^n$, where $B^{\ell}$ is the (classical) second fundamental form of $\p E_{\ell}$.
\end{oss}

Our first observation is that the class of oriented $n$-varifolds in $\Om\subset \R^{n+1}$ with curvature in $L^q$ is compact (in the oriented varifold topology) under locally uniform bounds on the masses and on the $L^q$-norms of the curvature coefficients, for $q>1$. This can be deduced as immediate consequence of the (more general) convergence of measure-function pairs in \cite{Hut} (using \cite[Proposition 4.2.4 and Proposition 4.4.2 (ii)]{Hut}). We prove it here directly in terms of convergence of (vector-valued) measures. (Here we work on the sequence $V^\ell$ under consideration; the general statement of this compactness property will be given in Proposition \ref{Prop:compactness_curvature_bounds_Riem}.)

\medskip

Consider the sequence of vector-valued measures $\vec{\mu}_\ell=\vec{W}^\ell V^\ell$, where $\vec{W}^\ell$ is the vector with $(n+1)^2$ entries $\{W^ \ell_{ia}\}$ for $(a,i)\in \{1, \ldots , n+1\} \times \{1, \ldots , n+1\}$ (for a fixed order on this set of indeces). By assumption, $\int_{U\times \R^{n+1}} |W_{ia}^\ell|^q d V^{\ell} \leq K$ for all $\ell$ and for all $i,a$, for the given $q>1$. 
Then the measures $\vec{\mu}_\ell$ have locally uniformly bounded total variations, $|\vec{\mu}_\ell|(U\times \R^{n+1})=\int_{U\times \R^{n+1}} |\vec{W}^\ell| dV^\ell \leq K^{1/q} C^{1/q'}$, where $C$ is the upper bound for the perimeters of $E_\ell$ in $U$ (provided by (\ref{eq:bounds_assumptions_continuous})) and $q'$ is the conjugate of $q$. By standard weak* compactness for Radon measures we can extract a (subsequential) limit $\vec{\mu}$. (We will use below the fact that the total variation is lower semi-continuous.) Extracting possibly a further subsequence, we also have the existence of an oriented integral varifold $V$ to which $V^{\ell}$ converge, thanks to the compactness theorem of Hutchinson, recalled in Section \ref{oriented_var}. Let $A\subset U\times \R^{n+1}$ be open and bounded; for $t>0$ let $A_t$ be the set of points in $A$ at distance $\geq t$ from $\p A$; let $u\in C^0_c(A)$ with $u\equiv 1$ on $A_t$ and $u\geq 0$. Then:

$$|\vec{\mu}|(A_t)\leq \liminf_{\ell\to \infty} |\vec{\mu_\ell}|(A_t) = \liminf_{\ell\to \infty}\int_{A_t} |\vec{W}^\ell| dV^\ell\leq  \liminf_{\ell\to \infty}\int_{U\times \R^{n+1}} u |\vec{W}^\ell| dV^\ell\leq$$ 
$$\liminf_{\ell\to \infty} \left(\int_{U\times \R^{n+1}} |\vec{W}^\ell|^q dV^\ell\right)^{1/q} \left(\int  u^{q'}dV^\ell\right)^{1/q'}\leq K^{1/q} \liminf_{\ell\to \infty} \left(\int u^{q'}dV^\ell\right)^{1/q'}$$
$$=K^{1/q} \left(\int u^{q'}dV \right)^{1/q'} \leq K^{1/q}\, V(A)^{1/q'}. $$
Then (since these are Radon measures) we get $|\vec{\mu}|(A)\leq K^{1/q}\, V(A)^{1/q'}$ for any Borel set $A$. In particular, $\vec{\mu}$ is absolutely continuous with respect to $V$ and is thus of the type $\vec{W} V$, with $\vec{W}\in L^1_{\text{loc}}(V)$. We will denote by $W_{ia}$ the components of $\vec{W}$.

The convergence $\vec{W}^\ell V^\ell \rightharpoonup \vec{W} V$ (weak* as vector-valued Radon measures) is equivalent to convergence as measure-function pairs of $(W^\ell_{i a}, V^\ell)$ to $(W_{ia}, V)$, in the sense of \cite[Proposition 4.2.4]{Hut}. Both notions amount to the fact that for every $(i,a)$
$$\int W_{ia}^{\ell} \phi dV^{\ell} \to \int W_{ia} \phi dV \text{ for every }\phi \in C^0_c(U\times \R^{n+1}).$$
Thus, for every $\phi\in C^1_c(U\times \R^{n+1})$, passing (\ref{eq:cuvature_identity_Vell}) to the limit, we obtain

\begin{equation}
\label{eq:cuvature_identity_limit}
\int \left((\delta_{is}-v_i v_s) \,D_s \phi + W_{i a}\, D^*_{a} \phi  - \phi (W_{rr} v_i + W_{ri} v_r)\right)d\,V=0.
\end{equation}
The oriented integral varifold $V$ thus has curvature in the sense of Definition \ref{Dfi:oriented_curvature_varifold}, with curvature coefficients $W_{ia} \in L^1_{\text{loc}}(V)$. In fact, we have $W_{ia} \in L^q_{\text{loc}}(V)$. (For the rest of our arguments in this section, we will only require that $W_{ia} \in L^1_{\text{loc}}(V)$.) Being convex, $y\in \R^{(n+1)^2} \to |y|^q$ is a supremum of countably many affine functions, $|y|^q = \sup_{m\in \N} (a_m + \vec{b}_m\cdot y)$ for a suitable choice of $a_m \in \R$, $\vec{b}_m\in \R^{(n+1)^2}$. Then we have, for each $m$, and for any $\psi\in C^0_c(U \times \R^{n+1})$, 
$$\int (a_m + \vec{b}_m \cdot \vec{W}) \psi \, dV =\lim_{\ell \to \infty}\int (a_m + \vec{b}_m \cdot \vec{W^\ell}) \psi \, dV^\ell \leq \liminf_{\ell \to \infty}\int |\vec{W}^\ell|^q \psi dV^\ell.$$ One can verify (this holds whenever a function is defined as a supremum) that $\int_A |\vec{W}|^q dV = \sup\left\{\sum_{m\in \N} \int (a_m + \vec{b}_m \cdot \vec{W})\phi_m \,dV\right\}$, where the supremum is taken over all possible choices of $\{\phi_m\}_{m\in\N}$ with $\phi_m\in C^0_c(A)$, $0\leq \phi_m\leq 1$, with $\text{spt}(\phi_m)$ pairwise disjoint and $A$ having compact closure. It then follows that $\int_A |\vec{W}|^q dV \leq \liminf_{\ell \to \infty}\int_A |\vec{W}^\ell|^q dV^\ell$.

\medskip

\begin{proof}[proof of Theorem \ref{thm:main}]
From Proposition \ref{Prop:further_prop_curv_coeff} (a), $-(W^\ell_{rr} v_i + W^\ell_{ri} v_r)$ in (\ref{eq:cuvature_identity_Vell}) is the lift ($V^{\ell}$-a.e.) of $\vec{H}^\ell\cdot \textbf{e}_i$, where $\vec{H}^\ell$ is the generalised mean curvature of $|\p^* E_\ell|$, which is $g_\ell \nu_\ell$ by assumption. Hence $V^{\ell}$-a.e.~$(x,v)$ we have $-(W^\ell_{rr} v_i + W^\ell_{ri} v_r)(x,v) = g_{\ell}(x) v_i$, the $i$-th component of $\vec{\mathscr{g}}_\ell(x,v)$. (Recall that $V^{\ell}=(\p^* E_\ell, \nu_{E_\ell}, 1,0)$, so the measurable choice of  orientation is $\nu_{E_\ell}$.) Then (\ref{eq:cuvature_identity_Vell}) can be equivalently written as
$$\int  \left((\delta_{is}-v_i v_s) \,D_s \phi + W^\ell_{i a}\, D^*_{a} \phi + \phi g_\ell v_i \right)d\,V^\ell=0.$$
Passing to the limit, since $g_\ell v_i \to gv_i$ in $C^0(U\times \R^{n+1})$ (and the total masses of $V^\ell$ are locally uniformly bounded), we obtain
\begin{equation}
 \label{eq:cuvature_identity_limit_with_g}
 \int  \left((\delta_{is}-v_i v_s) \,D_s \phi + W_{i a}\, D^*_{a} \phi + \phi g v_i\right)d\,V=0.
\end{equation}

Comparing (\ref{eq:cuvature_identity_limit_with_g}) and (\ref{eq:cuvature_identity_limit}), the arbitraryness of  $\phi\in C^1_c(U\times \R^{n+1})$ implies that $V$-a.e.~it holds $-(W_{rr} v_i + W_{ri} v_r) = g v_i$, the $i$-th component of $\vec{\mathscr{g}}(x,v)$.  Again by Proposition \ref{Prop:further_prop_curv_coeff} (a), $-(W_{rr} v_i + W_{ri} v_r)$ is the lift of $\vec{H}\cdot \textbf{e}_i$, where $\vec{H}$ is the generalised mean curvature of $\|V\|$. Hence $\vec{\mathscr{g}}(x,v)=g(x)\vec{v}=\vec{H}(x)$ (that is, $\vec{\mathscr{g}}$ descends to $\vec{H}$ and, in particular, the prescribed mean curvature feature holds for $V$). 

\medskip

Since $\textbf{c}(V^\ell)=\p\llbracket E_\ell \rrbracket$ and $\p\llbracket E_\ell \rrbracket \to \p \llbracket E\rrbracket$ (as currents), we conclude that $\textbf{c}(V)=\p \llbracket E \rrbracket$ and therefore $V=(\p^* E, \nu_E, 1, 0)+(\mathcal{R}, \textbf{n}, \theta, \theta)$, for an $n$-rectifiable set $\mathcal{R}$ and $\theta:\mathcal{R}\to \N\setminus \{0\}$ in $L^1_{\text{loc}}(\mathcal{H}^n\res \mathcal{R})$, and $\textbf{n}$ a measurable choice of unit normal on $\mathcal{R}$. Defining $R=\p^* E \cup \mathcal{R}$ we then write $V=(R, \xi, \theta_1, \theta_2)$, where $\xi$ is the measurable choice of orientation on $R$ taken to agree with $\nu_E$ on $\p^*E$ and with $\textbf{n}$ on $\mathcal{R}\setminus \p^* E$, and $(\theta_1, \theta_2)=(\theta+1,\theta)$ on $\p^* E \cap \mathcal{R}$, $(\theta_1, \theta_2)=(1,0)$ on $\p^* E \setminus \mathcal{R}$, $(\theta_1, \theta_2)=(\theta,\theta)$ on $\mathcal{R}\setminus \p^* E$.

By Lemma \ref{lem:mean_curv} (which can be applied since we can test in (\ref{eq:cuvature_identity_limit_with_g}) with $\phi$ independent of $v$) we find $\vec{H}(x)=\frac{\theta_1(x) g(x)\xi(x)-\theta_2(x) g(x)\xi(x)}{\theta_1(x)+\theta_2(x)}$, so that necessarily we must have $(\theta_1(x), \theta_2(x))=(1,0)$ for $\|V\|$-a.e.~$x\in \{g(x) \neq 0\}$ . In particular, possibly redefining $\mathcal{R}$ by removing a set of $\mathcal{H}^n$-measure $0$, $\mathcal{R}\subset \{g=0\}$.
Theorems \ref{thm:C^2_case} and \ref{thm:main} are proved.
 
\end{proof}

\begin{proof}[proof of Corollary \ref{cor:boundaries}]
We note that we can use $\{g=0\}$ as one of the $M_j$'s that give the $C^2$-rectifiability property for the rectifiable set $R$ such that $V=(R, \xi, \theta_1, \theta_2)$ (see Section \ref{geometric_coeff}). More precisely, given the cover $\{M_j\}_{j=1}^\infty$ as in Section \ref{geometric_coeff}, we can redefine $C_{1} = \{g=0\}$, $C_j = M_{j-1} \setminus \{g=0\}$ for $j\geq 2$ to yield a new countable $C^2$-cover $\{C_j\}_{j=1}^\infty$. We can then work with the new cover to obtain (fixing $j=1$) that $V$-a.e.~$(x,v) \in \tilde{C}_{1}$ the curvature coefficients $W_{ia}(x,v)$ of $V$ agree with the ``geometric'' curvature coefficients of $C_{1}$ and, in particular, the generalised mean curvature of $\textbf{q}_\sharp V$ agrees with $h_S$ almost everywhere on $S=\{g=0\}$. On the other hand, by Theorem \ref{thm:main}, almost everywhere on $\mathcal{R}$ we have that the generalised mean curvature of $\textbf{q}_\sharp V$ vanishes and that $\mathcal{R}\subset S$. Therefore $h_S = 0$ almost everywhere on $\mathcal{R}$. If $\mathcal{H}^n(\mathcal{R})>0$, there is a $\mathcal{H}^n$-measurable set $Z \subset S$ with $\mathcal{H}^n(Z)>0$ and such that $h_S=0$ on $Z$. Then there would be a point (of Lebesgue density $1$ for $Z$) where $h_S$ vanishes at infinite order, contradiction. So $\mathcal{H}^n(\mathcal{R})=0$, that is, $V=(\p^* E, \nu_E, 1, 0)$.
\end{proof}

\section{Riemannian setting and more general hypersurfaces}
\label{Riemannian}

Assume that $\Om$ is an oriented Riemannian manifold of dimension $n+1$ (as we are addressing a local problem, orientability can always be assumed without loss of generality). We will embed $\Om$ isometrically in and open set $U\subset \R^N$, for $N$ sufficiently large, and denote by $S(x)$ the matrix of orthogonal projection from $\R^N$ onto $T_x \Om\subset \R^N$. 
Following \cite{Hut}, $V$ is an oriented integral $n$-varifold in $\Om$ if it is an oriented integral $n$-varifold in $U$ and $\spt{V}\subset \Om$. This means that $V=(R, \vec{N}, \theta_1, \theta_2)$ for an $n$-rectifiable set $R\subset U$, a measurable choice of orientation $\vec{N}(x)$, with $\vec{N}(x)$ a unit simple $(N-n)$-vector whose span is orthogonal to $T_x R$ in $\R^N$, and two locally summable integer-valued functions $\theta_1, \theta_2$; moreover we have imposed $R\subset \Om$. Recall that, in particular, $V$ is a measure in $U\times G^o(n,N)$, where the second factor is the Grassmannian of oriented $n$-planes in $\R^N$, which naturally embeds in $\Lambda^n(\R^N) \equiv \R^{N \choose {n}}$.

To exploit the codimension-$1$ property of $V$ in $\Om$, we note that $\vec{N}$ determines uniquely a measurable function $\xi: R \to \mathbb{S}\Om$, where $ \mathbb{S}\Om$ is the unit sphere bundle of $\Om$, by the relation $\xi \wedge \vec{N}_{\Om} = \vec{N}$, where $\vec{N}_{\Om}$ is the smooth determination of unit simple $(N-(n+1))$-vector that orients the normal bundle of $\Om$ in $\R^N$. Note that $\mathbb{S}\Om \subset T\Om$ is realised as a smooth subbundle of $\Om \times \R^N$, with fiber at $x\in \Om$ given by $\mathbb{S}^{N-1} \cap T_x \Om$ (an $n$-sphere). This allows us to treat $V$ as a Radon measure in $U \times \R^N$, whose defining action on $\varphi\in C^0_c(U \times \R^N)$ is 
$$V(\varphi)=\int_R \left(\theta_1(x) \varphi(x, \xi(x)) + \theta_2(x) \varphi(x, -\xi(x))\right) d\mathcal{H}^n(x).$$
In fact, $V$ is supported in $\mathbb{S}\Om$. We write $V=(R, \xi, \theta_1, \theta_2)$.

\begin{Dfi}
 \label{Dfi:Riemannin_oriented_curvature}
We say that the oriented integral $n$-varifold $V=(R, \xi, \theta_1, \theta_2)$ in $\Om$ has curvature in $L^q_{\text{loc}}$ if there exist functions $W_{ia} \in L^q_{\text{loc}}(V)$, for $i,a \in \{1, \ldots, N\}$ such that the following identity holds for all $\varphi\in C^1_c(U \times \R^N)$:

$$\int \left((\delta_{sk} - v_s v_k) S_{kb} (D_s \varphi) + (D^*_a \varphi) W_{ba} +  (\delta_{ik} - v_i v_k) S_{kr} \overline{A}_{irb} \varphi +\right.$$ 
$$\left.-S_{rb} (W_{jr} v_k+ W_{jk}v_r) S_{kj} \varphi \right)\, dV=0,$$
where $\overline{A}_{ijk}=S_{ir}D_r S_{jk}$ are the curvature coefficients of $\Om$ in $\R^N$ (these are smooth functions of $x$) and $D_j$, $D^*_a$ denote the partial derivatives respectively with respect to the variable $x_j$ (in $U\subset \R^N$) and $v_a$ (in $\R^N$).
 
\end{Dfi}

Define, for $i,j\in \{1, \ldots, N\}$, the function $P_{ij}(x,v)  =(\delta_{ik} - v_i v_k) S_{kj}$ (which appears in the above identity). Note that when $v\in T_x \Om$ is unit, the matrix $P(x,v)$ represents the matrix of orthogonal projection from $\R^N$ onto the $n$-dimensional subspace orthogonal to $v$ in $T_x \Om$. When $v$ is a unit normal to a hypersurface $M\subset \Om$, then $P$ is the orthogonal projection from $\R^N$ onto $T_x M \subset \R^N$.

We point out the following relations, that involve the (smooth) curvature coefficients of $\Om$. Recall that (see \cite{Hut}) $\overline{A}_{irb}= \overline{A}_{ibr}= \overline{B}_{ir}^b+ \overline{B}_{ib}^r$, where $\overline{B}_{ij}^k = \langle \overline{\nabla}_{{\textbf{e}_i}^T} \textbf{e}_j^T, \textbf{e}_k^N\rangle$ stands for the second fundamental form coefficients of $\Om$ in $\R^N$, for $i,j,k\in\{1,\ldots, N\}$, with $\overline{\nabla}$ denoting differentiation in $\R^N$ and $T$ and $N$ denoting respectively tangential and normal components with respect to $\Om$. We have:

\begin{eqnarray*}
   S_{jr} \overline{B}_{db}^j =  S_{jr} \langle \nabla_{\textbf{e}_d^T} \textbf{e}_b^T, \textbf{e}_j^\perp\rangle =  S_{jr} \langle \overline{B}(\textbf{e}_d^T, \textbf{e}_b^T), \textbf{e}_j \rangle= S\left(\overline{B}(\textbf{e}_d^T, \textbf{e}_b^T)\right) \cdot \textbf{e}_r = 0 \cdot \textbf{e}_r =0; \\
   P_{ir} \overline{A}_{irb}=P_{ir} \overline{B}_{ir}^b+P_{ir} \overline{B}_{ib}^r = P_{ir} \overline{B}_{ir}^b, \text{ using } P_{ir}=P_{iy} S_{yr};\\
   S_{ab} S_{yr} \overline{A}_{irb}= S_{ab} S_{yr} \overline{B}_{ir}^b=  S_{yr} S_{ab} \overline{B}_{ir}^b = 0, \text{ hence also }P_{kb}P_{jd} \overline{A}_{dkj} =0.
  \end{eqnarray*}
The defining identity for $W_{ia}$ in Definition \ref{Dfi:Riemannin_oriented_curvature} is then equivalently written as
\begin{equation}
\label{eq:Riemannin_oriented_curvature_B}
\int \left(P_{sb} (D_s \varphi) + (D^*_a \varphi) W_{ba} +  P_{ir} \overline{B}_{ir}^{b} \varphi  -S_{rb} (W_{jr} v_k+ W_{jk}v_r) S_{kj} \varphi \right)\, dV=0. 
\end{equation}

\noindent \textit{Claim}. Let $M$ be a $C^2$ hypersurface in $\Om$, oriented by a choice of unit normal $\nu$. The associated oriented integral varifold $V_M=(M, \nu, 1,0)$ satisfies the conditions in Definition \ref{Dfi:Riemannin_oriented_curvature}.

\begin{proof}[proof of the claim]
Consider the vector field $\phi \textbf{e}_b$, for $b\in \{1, \ldots, N\}$ and $\phi\in C^1_c(\Om)$. Then the component tangential to $\Om$ is given by $\phi S_{rb} \textbf{e}_r$, and the component tangential to $M$ is given by $(\phi S_{rb} \textbf{e}_r)^t=\phi P_{rj}S_{rb} \textbf{e}_j$, so that $\text{div}_M (\phi S_{rb} \textbf{e}_r)^t = P_{sj}D_s(\phi P_{rj} S_{rb})$. Then the divergence theorem gives

$$\int P_{sr}S_{rb}(D_s \phi) + P_{sr}(D_s S_{rb}) \phi + P_{sj} S_{rb} (D_s P_{rj}) \phi=0.$$
Since $P\circ S = P$ (and with the usual notation $\delta_j=P_{is}D_s$) we rewrite the identity as

$$\int P_{sb}(D_s \phi) + (\delta_r S_{rb}) \phi + S_{rb} (\delta_j P_{rj}) \phi =0.$$
Substituting $P_{rj}=(\delta_{rk} - v_r v_k) S_{kj}$ and setting $W_{ik}=\delta_i(\nu \cdot \textbf{e}_k)=P_{si}D_s(\nu \cdot \textbf{e}_k)$,
$$\delta_j P_{rj}=(\delta_{rk} - v_r v_k)\delta_j S_{kj}-(\delta_j v_r)v_k S_{kj} - (\delta_j v_k)v_r S_{kj} = (\delta_{rk} - v_r v_k)P_{jd}S_{dr}D_r S_{kj}$$ $$-(\delta_j v_r)v_k S_{kj} - (\delta_j v_k)v_r S_{kj} =(\delta_{rk} - v_r v_k) P_{jd}\overline{A}_{dkj}-W_{jr} v_k S_{kj} - W_{jk}v_r S_{kj}.$$ 
Moreover, $\delta_r S_{rb}=P_{rd}D_d S_{rb} = P_{ir}S_{id} D_d S_{rb} = P_{ir}(\overline{\delta}_i S_{rb}) = P_{ir}\overline{A}_{irb}$. (We use the notation $\overline{\delta}_i=S_{id} D_d$.)
The identity becomes then
$$\int P_{sb}(D_s \phi) +P_{ir}\overline{A}_{irb} \phi + \underbrace{S_{rb} (\delta_{rk} - v_r v_k)}_{P_{bk}} P_{jd} \overline{A}_{dkj}\phi-S_{rb} (W_{jr} v_k+ W_{jk}v_r) S_{kj} \phi =0,$$
and the third term vanishes (by the relations given above).
Taking $\phi(x)=\varphi(x,\nu(x))$ for $\varphi(x,v)$ with $x\in \R^N$ and $v\in \R^N$, we obtain:
$$\int P_{sb}(D_s \varphi) + (D^*_a \varphi)\underbrace{P_{sb}D_s(\nu(x) \cdot \textbf{e}_a)}_{W_{ba}} + P_{ir}\overline{A}_{irb} \varphi-S_{rb} (W_{jr} v_k+ W_{jk}v_r) S_{kj} \varphi =0,$$
with $\varphi$ evaluated at $(x,\nu(x))$, which concludes the proof that $(M, \nu, 1,0)$ is an oriented integral varifold with curvature as in Definition \ref{Dfi:Riemannin_oriented_curvature}.
\end{proof}

Next, we note that, testing with $\varphi(x,v)=\phi(x)$ in (\ref{eq:Riemannin_oriented_curvature_B}), for $\phi\in C^1_c(\Om)$, we obtain 
$$\int \left(P_{sb} (D_s \varphi) +P_{ir} \overline{B}_{ir}^{b} \varphi -S_{rb} (W_{jr} v_k+ W_{jk}v_r) S_{kj} \varphi \right)\, dV=0.$$
All functions in the integrand, except $W_{ia}$ and $v_j$, are even in $v$ (some only depend on $x$). Set $\tilde{W}_b(x)$ to be the weighted average (as in Lemma \ref{lem:mean_curv}) of $S_{rb} (W_{jr} v_k+ W_{jk}v_r) S_{kj}$ at $\xi(x)$ and $-\xi(x)$, which is $L^1_{\text{loc}} (\textbf{q}_\sharp V)$. We then obtain

$$\int \left(P_{sb} (D_s \phi) + P_{ir} \overline{A}_{irb} \phi -\tilde{W}_b \phi \right)\, d (\textbf{q}_\sharp V)=0.$$
The first term in the integrand is $\text{div}_{\textbf{q}_\sharp V} (\phi \textbf{e}_b)$. Therefore the identity just obtained amounts to the first variation formula evaluated on the vector field $\varphi \textbf{e}_b$ and implies that $\textbf{q}_\sharp V$ has first variation in $L^1_{\text{loc}} (\textbf{q}_\sharp V)$. Explicitly, $\int \text{div}_{\textbf{q}_\sharp V} (\varphi \textbf{e}_b) +\varphi \textbf{e}_b \cdot \vec{H}_{\textbf{q}_\sharp V}^{\R^N}=0$, with $\vec{H}_{\textbf{q}_\sharp V}^{\R^N}=(P_{ir} \overline{A}_{irb} -\tilde{W}_b) \textbf{e}_b$ (the generalised mean curvature of $\textbf{q}_\sharp V$ in $\R^N$). On the other hand, testing with the vector field $\phi(x)S_{ab}(x)\textbf{e}_a = (\varphi \textbf{e}_b)^T$ gives

$$\int \left(P_{sa} (D_s (\phi S_{ab})) + \underbrace{P_{ir} S_{ab} \overline{A}_{ira}}_{=0} \phi - \underbrace{S_{ab} \tilde{W}_a}_{\tilde{W}_b} \phi \right)\, d (\textbf{q}_\sharp V)=0,$$
where we used respectively the relations given earlier and the definition of $\tilde{W}_{ab}$ for the two braced terms.
This identity is $\int \text{div}_{\textbf{q}_\sharp V} (\varphi \textbf{e}_b)^T +(\varphi \textbf{e}_b)^T \cdot \vec{H}_{\textbf{q}_\sharp V}^{\Om}=0$, where the generalised mean curvature $\vec{H}_{\textbf{q}_\sharp V}^{\Om}$ of $\textbf{q}_\sharp V$ in $\Om$ is given by $-\tilde{W}_b \textbf{e}_b$. Note that $\vec{H}_{\textbf{q}_\sharp V}^{\Om}$ is the projection onto $\Om$ of $\vec{H}_{\textbf{q}_\sharp V}^{\R^N}$.

\medskip

We have established in particular the following: if $V$ is as in Definition \ref{Dfi:Riemannin_oriented_curvature}, the first variation of $\textbf{q}_\sharp V$ in $\Om$ is represented by a function in $L^q_{\text{loc}}(\textbf{q}_\sharp V)$. It follows, thanks to \cite{Men}, that for $V$ as in Definition \ref{Dfi:Riemannin_oriented_curvature}, with $q\geq 1$, we have the existence of $\{M_j\}_{j=1}^\infty$ such that $\|V\|(U \setminus (\cup_{j=1}^\infty M_j))=0$ and each $M_j$ is a $C^2$-hypersurface in $\Om$. By redefining the $M_j$'s we can assume that each of them admits an orientation, with unit normal $\nu_{M_j}$. Then we can prove the following analogue of Propositions \ref{Prop:C2_coeff} and \ref{Prop:odd_unique}. 

\begin{Prop}
\label{Prop:C^2_coeff_Riem}
Let $V$ and $W_{ia}$ (for $i,a\in\{1,\ldots, N\}$) be as in Definition \ref{Dfi:Riemannin_oriented_curvature}, with $q\geq 1$, and $\{M_j\}_{j=1}^\infty$ as above. Then for each $j$, for $V$-a.e.~$(x,v) \in M_j$  we have $W_{ia}(x,v)=\delta_i (\nu_{M_j} \cdot \textbf{e}_a)$ when $v=\nu_{M_j}(x)$, and $W_{ia}(x,v)=-\delta_i (\nu_{M_j} \cdot \textbf{e}_a)$ when $v=-\nu_{M_j}(x)$. This is well-defined $V$-a.e. In particular, $W_{ia}$ is a.e. uniquely defined, and odd in $v$.
\end{Prop}

The proof of Proposition \ref{Prop:C^2_coeff_Riem} proceeds as in Section \ref{geometric_coeff}, using test functions 

$$\phi_a(x,v)= v_a - \frac{D_a f}{|\nabla f|}(x),$$
where $f:U \to \R$ is chosen so that $\{f=0\} \cap \Om = M_j$, $\nabla f = \nu_{M_j}$ on $M_j$. Such $f$ can be defined on $\Om$ and then extended to $U$ (since $\Om$ is a smooth submanifold and we can use tubular neighbourhood coordinates, possibly making $U$ smaller).

\medskip

We moreover have, arguing either directly as in Section \ref{proofs}, or by invoking \cite[Proposition 4.2.4 and 4.4.2(ii)]{Hut}, that the class of oriented $n$-varifolds in $\Om$ with curvature in $L^q_{\text{loc}}$ has the following compactness property.

\begin{Prop}
\label{Prop:compactness_curvature_bounds_Riem}
Let $\{V^\ell\}_{\ell\in \N}$ be a family of oriented integral $n$-varifolds in $\Om$ that satisfy the assumptions of Hutchinson's compactness theorem (see \cite{Hut}, or Section \ref{prelim}) and that moreover have curvature coefficients $W_{ia}^\ell\in L^q_{\text{loc}}(V^\ell)$ for $i,a\in \{1, \ldots, N\}$ (in the sense of Definition \ref{Dfi:Riemannin_oriented_curvature}) such that, for every compact $K$, $\sup_{\ell}\int_K |W^\ell_{ia}|^q\,dV^\ell<\infty$ for $q>1$. Then any oriented integral varifold $V$ in $\Om$ that is a limit point of $\{V^\ell\}$ has curvature $W_{ia}\in L^q_{\text{loc}}(V)$ and we have (subsequentially) $\vec{W}^{\ell_k} V^{\ell_k} \to \vec{W}  V$ (as vector-valued measures --- here $\vec{W}$ denotes the vector with entries $W_{ia}$ for a fixed ordering of $\{(i,a)\}_{i,a\in \{1, \ldots, N\}}$, and $\vec{W^\ell}$ is analogously defined; the same convergence can be expressed in terms of measure-function pairs $(W_{ia}^\ell, V^\ell)$ and $(W_{ia}, V)$ as in \cite{Hut}).
\end{Prop}

Since $W_{ia}$ are odd in $v$ (Proposition \ref{Prop:C^2_coeff_Riem}), $S_{rb} (W_{jr} v_k+ W_{jk}v_r) S_{kj}$ is the even lift of $\tilde{W}_b$, which, as we saw, is $\vec{H}_{\textbf{q}_\sharp V}^\Om \cdot \textbf{e}_b$. When the mean curvature of $V^\ell$ in $\Om$ is prescribed by $g^\ell\in C^0(\Om)$, we have $S_{rb} (W^\ell_{jr} v_k+ W^\ell_{jk}v_r) S_{kj} =- g^\ell v_b$. This is a continuous ambient function (in $U \times \R^N$, extending to a tubular neighbourhood of $\Om$). Repeating the arguments of Section \ref{proofs} we establish the following general compactness result:

\begin{thm}
\label{thm:overall}
Let $V^{\ell}=(R_\ell, \xi_\ell, \theta_1^\ell, \theta_2^\ell)$ be a sequence of integral oriented $n$-varifolds in an $(n+1)$-dimensional oriented Riemannian manifold $\Om$, with $\sup_{\ell\in\N}\|\textbf{q}_\sharp V^\ell\|(U)<\infty$ and $\sup_{\ell\in\N}M_U(\p \textbf{c}(V^\ell))<\infty$ for every $U\subset \subset \Om$. Assume that there is $q>1$ such that $V^{\ell}$ have curvatures in $L^q_{\text{loc}}$ (in the sense of Definition \ref{Dfi:Riemannin_oriented_curvature}), and that the curvature coefficients are locally uniformly bounded in the $L^q(V^\ell)$-norm. 
By the curvature assumption, $\delta(\textbf{q}_\sharp V^\ell)$ is represented by integration on $\|\textbf{q}_\sharp V^\ell\|$ of $\vec{H}^\ell \in L^q_{\text{loc}}(\|V^\ell\|)$. We assume that there exists a measurable choice of orientation $\nu_\ell$ on $R_\ell$ such that $\vec{H}^\ell = g_\ell \nu_\ell$ a.e. 

Then every varifold limit is of the type $\textbf{q}_\sharp V$ for an oriented integral varifold $V=(R, \xi, \theta_1, \theta_2)$ that has curvature in $L^q_{\text{loc}}$; moreover, there exists a measurable choice of orientation $\nu$ on $R$ such that $\vec{H} = g \nu$ a.e. (By Allard's theorem, $\delta(\textbf{q}_\sharp V)$ is represented by integration on $\|\textbf{q}_\sharp V\|$ of $\vec{H}\in L^q_{\text{loc}}(\|V\|)$, the generalised mean curvature.)
\end{thm}

Indeed, an analysis of the proof of Theorem \ref{thm:main} shows that the condition that the varifolds in question arise as boundaries of Caccioppoli sets is only used in the following steps. It permits to give an orientation to said varifolds, so that we can treat them as oriented integral varifolds. It permits the use of Hutchinson's compactness theorem, as it implies that the boundary in the sense of currents is $0$, in particular the boundary masses of the associated currents are locally uniformly bounded. Finally, once the prescribed mean curvature condition is established for the limit (in the paragraph following (\ref{eq:cuvature_identity_limit_with_g})), it permits the multiplicity-$1$ conclusion on $\{g\neq 0\}$. In Theorem \ref{thm:overall} the first two facts are already in the hypotheses and we do not aim to conclude the third, we only need to establish the preservation of the prescribed mean curvature condition: this follows from the arguments of Section \ref{proofs} up to the paragraph that follows (\ref{eq:cuvature_identity_limit_with_g}).

If, on the other hand, we add the hypothesis that $V^\ell$ are associated to boundaries of Caccioppoli sets, we follow the arguments of Section \ref{proofs} to the end and establish:

\begin{thm}
\label{thm:Riemannian}
Theorems \ref{thm:C^2_case} and \ref{thm:main}, and Corollary \ref{cor:boundaries}, hold with $\Om$ replaced by a Riemannian manifold (of dimension $n+1$). 
\end{thm}

\begin{oss}
For completeness, we point out that if $V$ is as in Definition \ref{Dfi:Riemannin_oriented_curvature} then:
\begin{itemize}
 \item[(1)] (as in Proposition \ref{Prop:further_prop_curv_coeff}) we have that $W_{ij}v_j=0$ and $W_{ij}=W_{ji}$, and in particular that $S_{rb} (W_{jr} v_k+ W_{jk}v_r) S_{kj}= W_{jk}S_{kj} v_b $.
 \item[(2)] (as in Section \ref{relation_with_Hut})  $\textbf{q}_\sharp V$ is a varifold with curvature in the sense of \cite{Hut}, where we view $\textbf{q}_\sharp V$ as measure in $U\times \R^{N^2}$, and $\R^{N^2}$ has coordinates $P_{ij}$, $i, j\in \{1, \ldots, N\}$. To see this, we take $\varphi\in C^1_c(U \times \R^{N^2})$ and use the chain rule in (\ref{eq:Riemannin_oriented_curvature_B}) with $\phi(x,v)=\varphi(x, P(x,v))$ with $P_{ij}(x,v)=(\delta_{ik}-v_iv_k) S_{kj}(x)$:
\begin{eqnarray*}
D_s \phi = D_s \varphi + D^*_{cd}\varphi D_s P_{cd} =  D_s \varphi + (D^*_{cd}\varphi) (\delta_{ck}-v_c v_k) (D_s S_{kd}),\\ 
D^*_a \phi = - (D^*_{cd}\varphi) S_{kd} (\underbrace{(D^*_a v_c)}_{\delta_{ac}} v_k + \underbrace{(D^*_a v_k)}_{\delta_{ak}} v_c) = - (D^*_{ad}\varphi) S_{kd} v_k -  (D^*_{cd}\varphi) S_{ad} v_c,
\end{eqnarray*}
so substituting we obtain
\begin{eqnarray*}
\int P_{sb} D_s \varphi+ (D^*_{cd}\varphi) (\delta_{ck}-v_c v_k)\underbrace{\overbrace{P_{jb}S_{js}}^{P_{bs}} (D_s S_{kd})}_{P_{jb}\overline{A}_{jkd}}-\\
-(D^*_{ad}\varphi) \underbrace{S_{kd} v_k}_{v_d}  W_{ba}-  (D^*_{cd}\varphi) S_{ad} v_c  W_{ba} + \left(P_{ir} \overline{A}_{irb} - S_{rb} (W_{jr} v_k+ W_{jk}v_r) S_{kj} \right) \phi=0. 
\end{eqnarray*}
Define $A_{bcd}=(\delta_{ck}-v_c v_k)P_{jb}\overline{A}_{jkd} - S_{kd}v_k  W_{bc}-S_{ad}v_c  W_{ba}.$ Then we only need
$$A_{dbd}=(\delta_{bk}-v_b v_k)P_{jd}\overline{A}_{jkd} - S_{kd} v_k  W_{db}-S_{ad}v_b  W_{da}$$
to coincide with the expression in parenthesis, that is,
$$P_{ir} \overline{A}_{irb} -S_{rb} W_{dr}v_k S_{kd} - \underbrace{S_{rb} v_r}_{v_b} W_{jk} S_{kj}.$$
When $W_{ia}$ are associated to a $C^2$ hypersurface oriented by $\nu$ we have $S_{rb} W_{dr} = S_{rb} \delta_d (\nu \cdot \textbf{e}_r)=S_{rb}P_{da} D_a (\nu \cdot \textbf{e}_r)=P_{da} D_a (\nu \cdot S_{rb}\textbf{e}_r)=\delta_d (\nu \cdot (S \textbf{e}_b))=\delta_d (\nu \cdot\textbf{e}_b)=W_{db}$. By Proposition \ref{Prop:C^2_coeff_Riem} we then have $S_{rb} W_{dr}=W_{db}$ holds $V$-a.e.
Finally, notice that $S_{bk}=P_{bk}+v_b v_k$ $V$-a.e., therefore $(\delta_{bk}-v_b v_k)P_{jd}\overline{A}_{jkd}=P_{jd}\overline{A}_{jbd}+(P_{bk}-S_{bk})P_{jd}\overline{A}_{jkd}=P_{jd}\overline{A}_{jbd}$. The proof of the claim is complete.
\end{itemize}

\end{oss}

A concrete instance of Theorem \ref{thm:overall} for hypersurfaces that do not arise as boundaries is the following:

\begin{thm}
\label{thm:immersions}
Let $g_\ell, g$ ($\ell\in \N$) be continuous functions in an $(n+1)$-dimensional Riemannian manifold $\Om$ such that $g_\ell \to g$ in $C^0_{\text{loc}}(\Om)$. Let $M_\ell=\iota_\ell(\Sigma_\ell)$ be two-sided properly $C^2$-immersed closed hypersurfaces in $\Om$ such that, for every open set $U\subset \subset \Om$, $\sup_{\ell\in \N} \mathcal{H}^n(M_\ell \cap U) <\infty$. (We denote by $\Sigma_\ell$ abstract $n$-manifolds and by $\iota_\ell$ the immersion.) Denote by $\nu_\ell$ a determination of unit normal for $M_\ell$ and assume that the mean curvature vector of $M_\ell$ is given by $g_\ell \nu_\ell$ for all $\ell\in \N$. (This is well-defined on $\Sigma_\ell$.) Denote by $B^{\ell}$ the second fundamental form of the immersion $M_\ell$ and assume that for every open set $U \subset \subset \Om$ we have 
$$\sup_{\ell\in \N} \int_{\Sigma_\ell \cap \iota_\ell^{-1}(U)} |B^{\ell}|^q \,d\mathcal{H}^n<\infty \text{ for some } q>1.$$
Then every varifold limit of $|M_\ell|$ is $\textbf{q}_\sharp V$ for an oriented integral varifold $V$ and the first variation of $\textbf{q}_\sharp V$ (as a functional on $C^1_c$-vector fields in $\Om$) is represented by integration (with respect to $\textbf{q}_\sharp V$) of $g \nu$, where $\nu$ is a measurable unit normal field on $\textbf{q}_\sharp V$ (a.e.~defined and orthogonal to $\textbf{q}_\sharp V$). 
\end{thm}

For example, let $g_\ell =g \equiv 2$ and $\Om = B^n_1(0) \times (-1,1)$, and consider $S = \{(x,y):x\in B_1^n(0), y =-\sqrt{1-|x|^2}\}$ and $M_\ell = (S + \frac{1}{\ell}\textbf{e}_{n+1}) \cup (S - \frac{1}{\ell}\textbf{e}_{n+1})$, where the choice of orientation $\nu_\ell$ is given by the unit vector whose scalar product with $\textbf{e}_{n+1}$ is positive. Then the varifold limit is $2|S|$ and $\nu$ is the upwards unit normal to $S$. Note that, with the given orientation, $M_\ell$ are not boundaries.

\begin{proof}
The two-sidedness assumption permits to consider the sequence of oriented integral varifolds $V^\ell=(M_\ell, \nu_\ell, 1,0)$. As the hypersurfaces $M_\ell$ are closed in $\Om$, the associated currents are cycles in $\Om$, hence Hutchinson's compacntess theorem applies to yield a (subsequential) limit $V$ in the sense of oriented integral varifolds. Then the proof of Theorem \ref{thm:main} can be repeated up to the point where we establish that $V$ satisfies (\ref{eq:cuvature_identity_limit_with_g}), from which the prescribed mean curvature condition is seen to hold for the limit and the conclusion of the theorem follows.
\end{proof}

\begin{oss}
In view of the hypotheses in Hutchinson's compactness theorem in \cite{Hut} (recalled in Section \ref{prelim}) one could adapt the notion of curvature varifolds with boundary, following Mantegazza's work \cite{Man}, to the oriented case and obtain a compactness result for $M_\ell$ immersed hypersurfaces-with-boundary with masses of the boundaries locally uniformly bounded in $\ell$.
\end{oss}

\begin{small}

\end{small}

\end{document}